\documentclass[12pt]{amsart}
\usepackage{amsmath,amsfonts,amsthm,amssymb,mathrsfs}
\usepackage{tocvsec2}

\usepackage{thmtools}
\renewcommand\thmcontinues[1]{continued}

\usepackage[margin=1in]{geometry}
\usepackage{tikz}
\usepackage{pgfplots}
\usepackage{comment}
\usepackage{xypic}
\usepackage{graphicx}
\usepackage[shortlabels]{enumitem}
\usepackage[bookmarks,colorlinks,citecolor=blue]{hyperref} 
\usepackage{tikz}

\newlength{\defbaselineskip} 

\theoremstyle{theorem}
\newtheorem{theorem}{Theorem}[section]
\newtheorem{corollary}[theorem]{Corollary}
\newtheorem{lemma}[theorem]{Lemma}
\newtheorem{proposition}[theorem]{Proposition}

\newtheorem{notation}[theorem]{Notation}

\newtheorem{pr}{Algorithm}
\newtheorem{fact}[theorem]{Fact}
\theoremstyle{definition} 
\theoremstyle{definition} \newtheorem{example}[theorem]{Example}

 \numberwithin{equation}{section}
 
 \theoremstyle{definition}
\newtheorem{con}[theorem]{Conjecture}
\newtheorem{definition}[theorem]{Definition}
\newtheorem{exm}[theorem]{Example}
\newtheorem{remark}[theorem]{Remark}

\def\RR{\mathbb{R}}
\def\ZZ{\mathbb{Z}}

\DeclareMathOperator{\rad}{Rad}

\DeclareMathOperator{\HS}{HS}
\DeclareMathOperator{\vol}{vol}

\def\a{\mathbb{A}}
\def\CC{\mathbb{C}}
\def\ob{\begin{obs}}
\def\kob{\end{obs}}
\def\dow{\begin{proof}}
\def\kdow{\end{proof}}

\def\tw{\begin{thm}}
\def\ktw{\end{thm}}
\def\hip{\begin{con}}
\def\khip{\end{con}}
\def\lem{\begin{lema}}
\def\klem{\end{lema}}
\def\ex{\begin{exm}}
\def\prog{\begin{pr}}
\def\kprog{\end{pr}}
\def\wn{\begin{cor}}
\def\kwn{\end{cor}}
\def\uwa{\begin{rem}}
\def\kuwa{\end{rem}}
\def\kex{\end{exm}}
\def\dfi{\begin{df}}
\def\kdfi{\end{df}}
\xyoption{line}
\def\fa{\begin{fact}}
\def\kfa{\end{fact}}

\newcommand{\facets}[1]{\overline{f}(#1)}
\def\mfl{\overline{\varphi}}
\def\ifl{{\varphi}}

\title{Many faces of symmetric edge polytopes}
\author{Alessio D'Al\`i}
\address{Mathematics Institute, University of Warwick, Coventry CV4 7AL, United Kingdom}
\email{Alessio.D-Ali@warwick.ac.uk}
\author{Emanuele Delucchi}
\address{Department of Mathematics, University of Fribourg, 1700 Fribourg, Switzerland}
\email{emanuele.delucchi@unifr.ch}
\author{Mateusz Micha\l ek}
\address{Max Planck Institute for Mathematics in the Sciences, 04103 Leipzig, Germany and \
Aalto University, Espoo, Finland and
Polish Academy of Sciences, Warsaw, Poland}
\email{wajcha2@poczta.onet.pl}
\thanks{}

\subjclass[2010]{Primary: 52B20; Secondary: 52B12, 13P10, 13P25, 05A15.} 

\begin{document}
\maketitle
\begin{abstract}
Symmetric edge polytopes are a class of lattice polytopes constructed from finite simple graphs. In the present paper we highlight their connections to the Kuramoto synchronization model in physics -- where they are called adjacency polytopes -- and to Kantorovich--Rubinstein polytopes from finite metric space theory. Each of these connections motivates the study of symmetric edge polytopes of particular classes of graphs. We focus on such classes and apply algebraic-combinatorial methods to investigate invariants of the associated symmetric edge polytopes.
\end{abstract}

\makeatletter 
\def\l@subsection{\@tocline{2}{0pt}{2pc}{6pc}{}} 
\makeatother

\setcounter{tocdepth}{1}
\tableofcontents
\section{Introduction}
Graphs and polytopes are among the most fundamental objects in mathematics. There are many constructions associating a polytope to a graph. In this article we focus on three of them: symmetric edge polytopes \cite{higashitani2019arithmetic,  higashitani2017interlacing, ohsugi2014centrally, ohsugi2012smooth, MHNOH}, adjacency polytopes \cite{chen2018counting, chen2018toric, chen2019directed} as well as Kantorovich--Rubinstein polytopes and Lipschitz polytopes \cite{vershik2015classification, delucchi2016fundamental, gordon2017combinatorics, jevtic2018cyclohedron, jevtic2018polytopal}. 

All of the described classes are intensely investigated, but each in a separate field. Symmetric edge polytopes, or more precisely the related Ehrhart polynomials, were first studied by algebraic number theorists \cite{BCKV, rod02} and later by algebraic combinatorialists. Adjacency polytopes appeared in the context of the Kuramoto model, describing the behavior of interacting oscillators \cite{kuramoto1975self}. The study of Kantorovich--Rubinstein polytopes of metric spaces is rooted in the research on the transportation problem. This class, together with their polar duals, Lipschitz polytopes, was brought to combinatorialists' attention by Vershik \cite{vershik2015classification} (see also \cite{MPV}), who suggested to study their face structure as a combinatorial invariant of metric spaces. We refer to Section \ref{sec:many} for the definitions and a more careful contextualization of the three classes of polytopes in their respective fields. To our knowledge, the connection between such research areas  does not appear to have been drawn yet.

The aim of our paper is precisely to exhibit such relations and exploit them in order to obtain new results about each of these families of polytopes.

In Section \ref{sec:many} we describe each of the above classes of polytopes and their applications. The main result in the section is Theorem \ref{thm:MainConnection}, which states that symmetric edge polytopes are exactly the adjacency polytopes and their coordinate linear cuts are fundamental polytopes. Hence, indeed, each class may be studied using the methods developed in the other fields. 

In Section \ref{sec:tools} we introduce our main tools. Some of them, like the description of facets of symmetric edge polytopes, have already appeared in one of the fields, but were not known in the other ones. Some, like the connection to the word counting problem and the Goulden--Jackson cluster method described in Section \ref{ssec:words}, are only unraveled in this article. The technique which proves to be the most useful to us is however the explicit description, initiated in \cite{higashitani2019arithmetic}, of some Gr\"obner bases of the associated toric algebras and related unimodular triangulations. These triangulations are very much sought for in the Kuramoto model context in order to develop homotopy techniques. However, to our knowledge, techniques involving Gr\"obner bases have not yet been used in that context.

Our main results are presented in Section \ref{sec:res}. Among them, we show how some important invariants of symmetric edge polytopes, such as the $h^*$-polynomial, change under basic graph constructions. This allows us to confirm the Nevo--Petersen conjecture \cite{nevo2011gamma} in some cases. We also provide very explicit descriptions of the polytopes for families of graphs that are interesting from the Kuramoto model perspective or for the study of metric spaces in the context of computational phylogenetics. Moreover, we give a formula for the number of integer points in polar duals of symmetric edge polytopes. This ties in with the active line of research on integer points in dual pairs of reflexive polytopes. We refer to the introduction to Section \ref{sec:res} for a more detailed statement of our results.

\subsection*{Acknowledgements}
MM would like to thank Piotr Pokora for pointing him to the article \cite{delucchi2016fundamental}, which initiated our research and Hidefumi Ohsugi for the reference \cite{ohsugi2019h} containing many interesting results. This project took shape in April 2019 at the Max Planck Institute for Mathematics in the Sciences in Leipzig, Germany. AD and ED are grateful to the institute for the generous hospitality. ED was supported by the Swiss National Science Foundation professorship grant PP00P2\_150552/1.

\settocdepth{subsection}
\section{Symmetric edge polytopes, Kuramoto model and metric spaces}
\label{sec:many}
In this section we describe the polytopes that are the main object of the article. Our focus is on their three different incarnations: as symmetric edge polytopes \ref{ssec:SEP}, as adjacency polytopes \ref{ssec:Kuramoto}  and as fundamental polytopes \ref{ssec:FundPoly}. As we will see, the first two are the same by definition. The main result of this section is  Theorem \ref{thm:MainConnection}, where we show that fundamental polytopes are precisely coordinate linear cuts of symmetric edge polytopes. 
\subsection{Symmetric edge polytopes}\label{ssec:SEP}
Symmetric edge polytopes are lattice polytopes associated to simple graphs. They were first introduced in \cite{MHNOH}. Let $G$ be a graph with vertex set $V$ and edge set $E$. Consider the lattice $\ZZ^V$ with basis elements $\mathbf{e}_v$ for $v\in V$.
\begin{definition}[Symmetric edge polytope]
The \emph{symmetric edge polytope} associated to a graph $G$ is:
\[\mathcal{P}_G:=\mathrm{conv}\left\{\mathbf{e}_v-\mathbf{e}_w, \mathbf{e}_w-\mathbf{e}_v : \{v,w\} \in E\right\} \subset \RR^V.\]
\end{definition}
\begin{example}[label = exm:K22]
Let $G$ be the complete bipartite graph $K_{2,2}$. The polytope $\mathcal{P}$ is three-dimensional with eight vertices.
\end{example}
Early interest in symmetric edge polytopes was spurred by the fact that certain families of polynomials sharing properties similar to Riemann's $\zeta$ function are Ehrhart polynomials of special symmetric edge polytopes \cite{BCKV, rod02}. Such polynomials have been in the focus of active research at least since P{\'o}lya's work \cite{polya1926bemerkung}, and they appear in different branches of mathematics, e.g., in the study of diophantine equations and Meixner polynomials \cite{kirschenhofer1999algebra}. Algebraic combinatorialists have formulated several conjectures and proved many theorems about zero loci of the Ehrhart polynomials of symmetric edge polytopes \cite{higashitani2019arithmetic,  higashitani2017interlacing, ohsugi2014centrally, ohsugi2012smooth}. In particular, $\mathcal{P}_G$ is always a reflexive polytope, i.e., its dual is also a lattice polytope. Reflexive polytopes, through toric geometry, play an important role in mirror symmetry \cite{Batyrev} and the study of Gorenstein toric varieties \cite{kasprzyk2012fano}.

Several methods turned out to be very useful in the study of symmetric edge polytopes, e.g., the theory of interlacing polynomials \cite{higashitani2017interlacing}. In our work the connection to toric ideals and in particular to Gr\"obner bases played the most important role: we describe these techniques in Section \ref{ssec:GB}. 
Below we present one of the most prominent examples of symmetric edge polytope.
\begin{example}\label{ex:cross} Let $G$ be any tree with $n$ vertices. The symmetric edge polytope $\mathcal{P}_G$ is unimodularly equivalent to the convex hull of the vectors $\{\mathbf{e}_i-\mathbf{e}_{i+1}, 1\leq i\leq n-1\}$ and to the cross-polytope, i.e.~the convex hull of the signed basis vectors of the lattice $\{\pm \mathbf{e}_i, 1\leq i\leq n-1\}$. It has $2^{n-1}$ facets corresponding to orthants of $\RR^{n-1}$. All roots of the Ehrhart polynomial lie on a line $\{z:\texttt{Re}(z)=-\frac{1}{2}\}$, cf.~\cite[Example 3.3]{higashitani2017interlacing}.
\end{example}  
\subsection{The Kuramoto model}\label{ssec:Kuramoto}

The Kuramoto model describes the behaviour of interacting oscillators. Classically, these are modeled by differential equations in the phase angles of the oscillators, with constant coefficients (one for each pair of oscillators, determining the strength of the coupling) and one constant frequency for each oscillator \cite{kuramoto1975self}.

The oscillators are often represented by vertices of a graph $G$. An edge of $G$ joins two given oscillators if they interact directly. The Kuramoto model found a place in many applications, e.g.~in physics, biology, chemistry, engineering and even social networks \cite{dorfler2014synchronization}. One of the fundamental problems is to understand the steady states of the system (that is: when the differentials of the phase angles vanish). 

In a recent series of papers \cite{chen2019directed,chen2018toric,chen2018counting}, a new method for the study of the steady states by means of \emph{adjacency polytopes} associated to $G$ has been put forward. It turns out that, by definition, the adjacency polytope of a graph $G$ is exactly the symmetric edge polytope $\mathcal{P}_G$. 
A main feature of the new approach is a change of variables in the original system of equations, which reduces the problem to solving (Laurent) polynomial equations in the algebraic torus $(\CC^*)^n$. In order to bound the number of solutions, the authors apply the theory developed by Kushnirenko and Bernstein \cite{bernvstein1975number, kushnirenko1976newton}, cf.~\cite{kaveh2012newton}, where the normalized volume of the adjacency polytope plays a central role.
 The authors of \cite{chen2018counting} refer to it as the \emph{adjacency polytope bound} and prove that in many cases it gives much better estimates of the maximal number of possible solutions than those obtained with previously available methods. Furthermore, we note that regular triangulations of adjacency polytopes are central in the homotopy methods developed in this context \cite{chen2018toric}. 

So far this technique has been carried out in practice only for very special families of graphs. We believe that one of the main reasons is that adjacency polytopes, alias symmetric edge polytopes, are quite complicated combinatorial objects themselves and providing triangulations and facet descriptions is nontrivial. One of the missing tools seemed to be the theory of Gr\"obner bases, which we apply successfully here. In particular, we provide explicit results about triangulations and volumes for many families of graphs. Our results are not exhaustive and we believe that our method can be successfully applied to many other graphs. 

\subsection{Fundamental polytopes of metric spaces}\label{ssec:FundPoly}
In \cite{vershik2015classification} Vershik proposed a combinatorial study of finite metric spaces via the face structure of certain polytopes that arise in the context of Kantorovich and Rubinstein's work on the transportation problem. Let $(X,d)$ be a finite metric space.
\begin{definition}[Kantorovich--Rubinstein polytope]
Let $\RR^X$ be the real vector space with the basis $\mathbf{e}_j$ for $j\in X$. The {\em Kantorovich--Rubinstein} polytope $KR(X,d)$ is the following convex hull:
\[KR(X,d):=\mathrm{conv}\left\{\frac{\mathbf{e}_i-\mathbf{e}_j}{d(i,j)}:i,j\in X\right\}.\]
\end{definition}

The dual of the KR polytope is known as the \emph{Lipschitz polytope} as its points represent Lipschitz functions on $(X,d)$ with Lipschitz constant equal to one. Explicitly, it is defined by
\[LIP(X,d)=\{x\in\RR^X:\sum_ix_i=0, x_i-x_j\leq d(i,j)\ \forall_{i,j\in X}\}.\]
The name {\em fundamental polytope} of $(X,d)$ has been used in \cite{delucchi2016fundamental,vershik2015classification} in order to refer to the polytope $KR(X,d)$. The problem then \cite[Problem 1.1]{vershik2015classification} is to relate the number of faces (and their incidences) of the fundamental polytope to the structure of the metric space.

This line of research has been taken up in the literature from different points of view, see \cite{delucchi2016fundamental,gordon2017combinatorics, jevtic2018cyclohedron, jevtic2018polytopal}. Face numbers of fundamental polytopes were computed for a class of ``generic'' metric spaces by Gordon and Petrov \cite[Introduction]{gordon2017combinatorics}, and in \cite{delucchi2016fundamental} for ``tree-like'' metric spaces. In order to motivate the name of the latter class, recall the following classical construction. Let $G=(V,E)$ be a graph and consider a weight function $w:E\to \mathbb R_+$. This data defines a metric space on any subset $V_1\subseteq V$ of the vertices of $G$ by letting the distance between any two vertices in $V_1$ be the minimum weight of a path in $G$ that joins them (the weight of a path being the sum of the weights of its edges). Such weighted graphs are used for instance in phylogenetic analysis in order to encode genetic dissimilarity data. In this context, the structure of $G$ reflects the fact that such data are given as a set of weighted bipartitions of the set of vertices. The resulting graphs are called {\em splits networks} and the metric spaces they represent are called {\em split-decomposable}, see e.g.~\cite{Huson} and cf.~also Section \ref{ssec:splits}. The most basic examples of splits graphs are trees (i.e., acyclic, connected graphs), and the associated metric spaces are called {\em tree-like}.  Somewhat surprisingly, even in this basic case the formulas for the face numbers found in \cite{delucchi2016fundamental}, although explicit, are quite involved. No explicit formulas are found in the literature beyond the above-mentioned cases of trees and ``generic'' spaces. This level of complexity is less surprising once one considers our first Theorem \ref{thm:MainConnection}, showing that that a special class of KR-polytopes of metric spaces consists of linear sections of symmetric edge polytopes.

\subsection{Connections}\label{ssec:MainConnection}
We next exhibit the connections between fundamental polytopes and symmetric edge polytopes. Let $G=(V,E)$ be a simple, connected graph.
For a subset $V_1\subset V$ we will use the notation $\mathcal{P}_G \cap \mathbb{R}^{V_1}$ to denote the intersection of $\mathcal{P}_G$ with the linear space $\{x_i = 0 \mid i \notin V_1\}$. On $V_1$ we consider the metric $d$ defined as above with the trivial weighting $w(e)=1$ for all $e\in E$ and we set
$$
\mathcal{K}_{G,V_1} := KR(V_1,d).
$$

\begin{theorem}\label{thm:MainConnection}
For any subset of vertices $V_1\subset V$ we have:
\[\mathcal{P}_G \cap \mathbb{R}^{V_1} = \mathcal{K}_{G, V_1}.\]
\end{theorem}
Let us record here some remarks that will come in handy in the proof of Theorem \ref{thm:MainConnection}.

\begin{remark} \label{NonMinimalPathRemark}
One has that $\frac{\mathbf{e}_i - \mathbf{e}_j}{d'} \in \mathcal{K}_{G, V_1}$ whenever $i, j \in V_1$ and $d' \geq d_G(i, j)$. This holds because $\mathbf{0},\frac{\mathbf{e}_i - \mathbf{e}_j}{d_G(i,j)} \in \mathcal{K}_{G, V_1}$ and the point $\frac{\mathbf{e}_i - \mathbf{e}_j}{d'}$ lies on the line segment between $\mathbf{0}$ and $\frac{\mathbf{e}_i - \mathbf{e}_j}{d_G(i,j)}$.
\end{remark}

\begin{remark} \label{GraphRepRemark}
Let $\mathbf{q} \in \mathcal{P}_G$. Then $\mathbf{q}$ is a convex combination of $\pm (\mathbf{e}_i - \mathbf{e}_j)$ for $\{i,j\} \in E(G)$. We can associate with $\mathbf{q}$ a directed weighted graph: if $\mathbf{e}_i-\mathbf{e}_j$ and $\mathbf{e}_j-\mathbf{e}_i$ appear in $\mathbf{q}$ with weights $\omega_1 < \omega_2$, then we will draw the directed edge going from $i$ to $j$ with weight $\omega_2 - \omega_1 > 0$. The weights in such a directed weighted graph will always be nonnegative and sum up to at most one. (Since $\mathbf{0} \in \mathcal{P}_G$, we can still regard this as a convex combination.) Moreover, we will always assume without loss of generality that such a directed graph is acyclic. Indeed, if a directed cycle $i_1 \xrightarrow{\lambda_1} i_2 \xrightarrow{\lambda_2} \cdots \xrightarrow{\lambda_{m-1}} i_m \xrightarrow{\lambda_m} i_1$ exists, we can obtain a new representation of $\mathbf{q}$ by subtracting $\mu := \min_{j}\lambda_j$ from every edge in the cycle. At least one of the new weights $(\lambda_j-\mu)$ is then zero, and thus the corresponding edge has been deleted from the directed graph. 
\end{remark}

\begin{remark} \label{BalancedRemark}
If $\mathbf{q} \in \mathcal{P}_G \cap \mathbb{R}^{V_1}$, consider a graphical representation of $\mathbf{q}$ as in Remark \ref{GraphRepRemark}. Then for each vertex not in $V_1$ the sum of the weights of the incoming edges must equal the sum of the weights of the outcoming ones. We will say that these vertices are \emph{balanced}.
\end{remark}

\begin{proof}[Proof of Theorem \ref{thm:MainConnection}]
By construction, $\mathcal{K}_{G, V_1}$ is contained in $\mathbb{R}^{V_1}$. To see that it is also contained in $\mathcal{P}_G$, pick two vertices $i, j\in V_1$ and consider a minimal path in $G$ between them, say $i = i_0 \to i_1 \to \ldots \to i_d = j$. Then \[\frac{1}{d}(\mathbf{e}_i-\mathbf{e}_j) = \frac{1}{d}(\mathbf{e}_{i_0}-\mathbf{e}_{i_1}) + \frac{1}{d}(\mathbf{e}_{i_1}-\mathbf{e}_{i_2}) + \ldots + \frac{1}{d}(\mathbf{e}_{i_{d-1}}-\mathbf{e}_{i_d}),\] which is a convex combination of points in $\mathcal{P}_G$.

Let us now prove that $\mathcal{P}_G \cap \mathbb{R}^{V_1}$ is contained into $\mathcal{K}_{G, V_1}$. Pick a point $\mathbf{q}$ in $\mathcal{P}_G \cap \mathbb{R}^{V_1}$ and consider a representation of it as a weighted directed acyclic graph as in Remark \ref{GraphRepRemark}.

First note that, if the representation of $\mathbf{q}$ admits no vertices in $V_1$ with at least one outcoming edge, then there are no edges at all and hence $\mathbf{q} = \mathbf{0}$. To see this, assume by contradiction that an edge exists. Then its source $v$ must be not in $V_1$. Since such vertices are balanced by Remark \ref{BalancedRemark}, there must be another edge whose target is $v$, and so on. Going on with this process we either create a cycle or meet a vertex in $V_1$, and both of these possibilities are forbidden.

Now assume that $\mathbf{q} \neq \mathbf{0}$ and pick a vertex $i_1\in V_1$ with an outcoming edge. Following this edge we either get to a vertex in $V_1$ or in $V\setminus V_1$. In the latter case, by balancedness there must exist an outcoming edge, which we then follow. After a finite number of steps we meet a vertex $j_1\in V_1$ and we have thus created a path $\tilde{p}_1$ of length $d_1$ between the two vertices $i_1$ and $j_1$ in $V_1$. Call $\mu_1$ the smallest weight to be found on the path $\tilde{p}_1$, and call $p_1$ the path obtained from $\tilde{p}_1$ by replacing all weights by $\mu_1$. We can now modify our directed graph by ``subtracting $p_1$'', i.e.~by substituting all weights $\lambda$ of the edges of $\tilde{p}_1$ by $\lambda - \mu_1$. This yields a new weighted directed acyclic graph with a strictly smaller number of edges and where vertices not in $V_1$ are still balanced. Such a graph is a representation of the point $\mathbf{q} - \mu_1(\mathbf{e}_{j_1}-\mathbf{e}_{i_1})$. We now repeat the whole procedure until, after a finite number $N$ of steps, we have the edgeless graph, which corresponds to the point $\mathbf{0}$. Summing all the contributions we then get a decomposition
\begin{equation} \label{PathDecompositionEq}
\mathbf{q} = \sum_{k=1}^{N}\mu_k (\mathbf{e}_{j_k} - \mathbf{e}_{i_k}) = \sum_{k=1}^{N}d_k \mu_k \frac{\mathbf{e}_{j_k} - \mathbf{e}_{i_k}}{d_k},
\end{equation}
where $\frac{\mathbf{e}_{j_k} - \mathbf{e}_{i_k}}{d_k} \in \mathcal{K}_{G, V_1}$ by Remark \ref{NonMinimalPathRemark}, and $\sum_{k=1}^{N}d_k \mu_k$ equals the sum (which is at most one) of the original weights in the directed graph associated with $\mathbf{q}$. Since $\mathbf{0} \in \mathcal{K}_{G, V_1}$, we can regard $\sum_{k=1}^{N}d_k \mu_k \frac{\mathbf{e}_{j_k} - \mathbf{e}_{i_k}}{d_k}$ as a convex combination of points in $\mathcal{K}_{G, V_1}$, as desired.
\end{proof}
\begin{corollary}
The Lipschitz polytope $LIP(G,V_1):=(\mathcal{K}_{G, V_1})^*$ is the projection of the lattice polytope $(\mathcal{P}_G)^*$ by the map dual to the inclusion $\RR^{V_1}\subset\RR^V$.  
\end{corollary}

\settocdepth{section}
\section{Algebraic and combinatorial methods}\label{sec:tools}
In this section we describe the various methods we will apply to study the symmetric edge polytopes. 
\subsection{Facets of symmetric edge polytopes}\label{ssec:faces}
A fundamental theorem of polytope theory states that a subset of Euclidean space is the convex hull of a finite set of points (i.e., a polytope) if and only if it is the intersection of a finite number of closed halfspaces. A {\em face} of a polytope $P$ is any subset of $P$ that can be obtained as the intersection of $P$ with a zero set of an affine linear form that is nonnegative on $P$ \cite[\S 2.1]{ziegler}. The set of all faces of $P$, partially ordered by inclusion, is the {\em face poset} $\mathscr F(P)$ of $P$ (notice that $\mathscr F(P)$ has a unique maximal element, $P$ itself, and a unique minimal element, the empty face $\emptyset$).  The \emph{$f$-vector} of a $d$-dimensional polytope is $f(P)=(f_0,f_1,\dots, f_{d-1},f_d)$, where $f_i$ is the number of $i$-dimensional faces. By convention we often set $f_{-1}=1$, counting the face corresponding to the empty set. We define the \emph{$f$-polynomial} by 
\begin{equation}\label{eq:fpoly}
f(t)=\sum_{j=0}^{d+1} f_{j-1}t^j. 
\end{equation}

For a polytope $P$ of dimension $d$,  a face of dimension $d-1$ is called a {\em facet}, and we will write $\facets{P}$ for the number of facets of a polytope $P$ when the dimension of $P$ needs not be specified. If $P$ is full-dimensional, i.e.~$P\subset \RR^d$, then each facet $F$ determines, uniquely up to a positive constant, a linear function $l\in (\RR^d)^*$ such that $l$ is constant on $F$ and $l(f)\leq l(p)$ for any $f\in F$ and $p\in P$. 

Symmetric edge polytopes are not full dimensional as they are contained in the hyperplane $\RR^V_0\subset\RR^V$ defined by setting the sum of all coordinates equal to zero. If the graph $G$ is connected, then $\mathcal{P}_G$ is full dimensional in $\RR^V_0$. The inclusion $\RR^V_0\subset\RR^V$ gives rise to the dual surjection $(\RR^V)^*\rightarrow(\RR^V_0)^*$. Hence, linear forms in $(\RR^V_0)^*$ may be represented by elements of $(\RR^V)^*$, which in turn can be identified with functions $f\colon V\rightarrow \RR$. Such a representation is not unique: more precisely, two functions $f_1,f_2$ are identified as elements of  $(\RR^V_0)^*$ if they differ by a constant, i.e.~$f_1(v_1)-f_2(v_1)=f_1(v_2)-f_2(v_2)$ for every $v_1,v_2\in V$. The representation becomes unique if, for example, we fix $v\in V$ and assume $f(v)=0$. Using this identification one can give an explicit description of when $f$ corresponds to a facet of $\mathcal{P}_G$.
\begin{theorem}\label{thm:facets}\cite[Theorem 3.1]{higashitani2019arithmetic}
Let $G=(V,E)$ be a finite simple connected graph. Then $f\colon V\rightarrow \ZZ$ is facet-defining if and only if
\begin{itemize}
\item[(i)] for any edge $e=uv$ we have $|f(u)-f(v)|\leq 1$, and
\item[(ii)] the subset of edges $E_f=\{e=uv\in E \colon |f(u)-f(v)|=1\}$ forms a spanning connected subgraph of $G$.
\end{itemize}
\end{theorem}
The vertices of the polytope that belong to the facet defined by $f$ correspond to those directed edges $(u,v)\in E$ for which $f(v)-f(u)=1$.
\begin{notation}\label{not:G_F}
For a symmetric edge polytope $\mathcal{P}_G$ and a facet $F$ associated with the labeling $f\colon V(G) \to \ZZ$, we will denote by $G_F$ the (oriented) subgraph of $G$ obtained by selecting only those edges $(u, v)$ for which $f(v) - f(u) = 1$.
\end{notation}
\begin{example}[continues = exm:K22]
Let $G=K_{2,2}$. We may consider a function $f$ that takes value zero on two vertices that are not joined and value one on two other vertices. This function satisfies the assumptions of Theorem \ref{thm:facets} and hence it determines a facet $F$. We have four edges $(u,v)$ in $G_F$ (i.e.~for which $f(v)-f(u)=1$), namely $u$ must belong to the part on which $f$ vanishes and $v$ to the other part. These four edges correspond to four vertices:
$$(-1,0,1,0);(-1,0,0,1);(0,-1,1,0);(0,-1,0,1).$$
This is a square, which is a facet of the polytope described in Example \ref{exm:K22}. The graph $G_F$ is presented in Figure \ref{Fig:K22} at the end of Section \ref{ssec:GB}.
\end{example}
As we have seen in Section \ref{ssec:Kuramoto} the normalized volume is one of the crucial invariants we would like to understand. It is also important in the study of Ehrhart polynomials, as it is (up to the factorial of the dimension of the polytope) the coefficient of the leading term and equals the degree of the corresponding toric variety. 

One of the methods to compute volumes of reflexive polytopes is to find, if possible, a unimodular triangulation of the boundary, i.e.~a subdivision into lattice simplices of minimal possible volume. This induces a subdivision of the polytope itself into simplices, by extending each simplex on the boundary by $\bf{0}$. 
As each facet of a reflexive polytope is of lattice distance one to ${\bf 0}$ each of the new simplices is also unimodular. Hence, the normalized volume of the polytope equals the number of such simplices. The following result provides a combinatorial description of simplices in a given facet. 
\begin{corollary}\label{cor:simpinfacet}\cite[Corollary 3.2]{higashitani2019arithmetic} 
The unimodular simplices contained in a facet of $\mathcal{P}_G$ represented by a function $f$ correspond exactly to (undirected) spanning trees that are subgraphs of $G_F$.
\end{corollary}
\begin{example}[continues = exm:K22]
We have four possible simplices in the facet corresponding to Figure \ref{Fig:K22}. Indeed, removing any edge from $G_F$ gives us a spanning tree. This corresponds to the fact that if we remove one point of the square facet $F$ the convex hull of the other vertices is a simplex. 
\end{example}
The previous example shows how to detect all of the simplices of a facet. This is of course not a triangulation. For instance, among the four possible simplices that are convex hulls of $3$-element subsets of the vertices of a square we have two possible choices of pairs of simplices that triangulate the square. A very powerful tool to obtain such triangulations is the theory of Gr\"obner bases, which we describe next.

\subsection{Gr\"obner bases and triangulations}\label{ssec:GB}
Toric geometry studies relations among finite configurations of lattice points and special algebraic varieties. The main idea is to identify a point $\bf{a}\in\ZZ^n$ with a Laurent monomial ${\bf{x}}^{\bf{a}}:=x_1^{a_1}\cdots x_n^{a_n}$. Then a subset of points $\mathcal{S}\subset\ZZ^n$ is identified with the closure of the image of the map given by the corresponding $m:=|\mathcal{S}|$ many monomials. This closure is, by definition, a toric variety. Its defining ideal $I_\mathcal{S}\subset \CC[y_1,\dots,y_m]$ is generated by binomials, i.e.~polynomials of the type ${\bf{y}}^{\bf{a}}-{\bf{y}}^{\bf{b}}$ for some ${\bf{a}},{\bf{b}}\in\ZZ_{\geq 0}^m$. A case of particular interest is when $\mathcal{S}$ is the set of lattice points in a polytope with vertices contained in $\{1\}\times\ZZ^n$.
There are now many textbooks devoted to toric geometry and its interplay with discrete convex geometry \cite{bruns2009polytopes,cox2011toric,fulton1993introduction,Stks}. The reader may also find a short introduction in \cite{MTor} and \cite[Chapter 8]{ourbook}.
\begin{example}[label = exm:toricquadric]
Let $\mathcal{S}=\{(1,0,0),(1,1,0),(1,0,1),(1,1,1)\}\subset \{1\} \times \ZZ^2$ contain all lattice points of the unit square. The associated toric variety is the closure of the image of the map:
\[(\mathbb C^*)^3 \to \mathbb C^4,\quad (t,x_1,x_2)\mapsto (t,tx_1,tx_2,tx_1x_2).\]
Here the domain is the algebraic torus $(\CC^*)^3$, as some of the monomials in general may have negative exponents. The ideal $I_\mathcal{S}\subset\CC[y_1,\dots,y_4]$ is generated by $y_1y_4-y_2y_3$.
\end{example}
Toric geometry offers both combinatorial tools to study toric varieties and algebraic tools to study lattice polytopes. Here we focus on the latter connection and outline applications of Gr\"obner bases to the study of triangulations of polytopes, and our main reference is \cite{Stks}. Let $\prec$ be a term order on monomials in $\CC[{\bf y}]$. For any $f\in\CC[{\bf y}]$ the largest monomial appearing in $f$ with a nonzero coefficient is called the \emph{initial term} of $f$ and is denoted by $\texttt{in}_\prec(f)$.
For any ideal $I\subset\CC[{\bf y}]$, the initial ideal of $I$ is $\texttt{in}_\prec(I):=\langle\texttt{in}_\prec(f):f\in I\rangle$, the ideal generated by all initial terms of elements of $I$. By definition, this is always a monomial ideal. We note that in general it is not enough to take initial terms of generators of $I$ in order to obtain the generators of the initial ideal $\texttt{in}_\prec(I)$.
\begin{definition}[Gr\"obner basis]
A finite set $G$ of generators of $I$ is called a \emph{Gr\"obner basis} if the initial terms of $G$ generate $\texttt{in}_\prec(I)$.
\end{definition}
For any ideal, Gr\"obner bases exist and may be computed, e.g.~by using the Buchberger algorithm \cite[2.7]{cox1997ideals}. 
\begin{example}[continues = exm:toricquadric]
 First consider the lexicographic term order $y_4\prec y_3\prec y_2\prec y_1$. Then the initial ideal of $I=(y_1y_4-y_2y_3)$ is $\texttt{in}_\prec(I)=(y_1y_4)$.

If we keep a lexicographic term order, but change the order of the variables so that $y_1\prec y_2\prec y_4\prec y_3$, then the initial ideal of $I$ is $\texttt{in}_\prec(I)=(y_2y_3)$. 

In fact, for any other term order the initial ideal will be one of the two described above.
\end{example}
In what follows our ideals will always be homogeneous and our term orders compatible with respect to the degree, i.e.~monomials of higher degree will be larger in the term order than those of smaller degree. One of the most often used term orders is the degree reverse lexicographic order \emph{degrevlex}, see \cite[p.\ 8]{ourbook}.

\begin{definition}
From now we will take $\mathcal{S}$ to be the set of lattice points of a polytope $P$ with vertices in $\{1\}\times\ZZ^n$, and we will call $I_P$ the associated ideal.
\end{definition}

Algebraically, the initial ideal is the best monomial approximation of the starting ideal sharing a lot of important invariants, such as the dimension, the degree and the Hilbert polynomial. Geometrically, the associated variety is a flat deformation of the original toric variety to a (possibly nonreduced) union of coordinate subspaces. Combinatorially,  $\texttt{in}_\prec(I_P)$ is a triangulation of the starting polytope. 

We explain this last statement in detail. Recall that the variables $y_i$ correspond to lattice points of $P$. Hence, (the radical of) a monomial in the $y_i$'s corresponds to a subset of lattice points in $P$. We define a family $\Delta$ of sets of lattice points in $P$ as follows: a set $Q$ of lattice points belongs to $\Delta$ if and only if the product of the variables $y_i$ corresponding to points in $Q$ does not belong to the radical of $\texttt{in}_\prec(I)$. As $\texttt{in}_\prec(I)$ is an ideal, the family $\Delta$ is closed under taking subsets, i.e.~is a simplicial complex. We refer to elements of $\Delta$ as faces. We note that the minimal nonfaces of $\Delta$ correspond to generators of $ \rad (\texttt{in}_\prec(I))$. 
\begin{example}[continues = exm:toricquadric]
In the considered case $y_1,y_4$ correspond to two diagonal points in the square and $y_2,y_3$ correspond to the other two diagonal points. We focus on the case $\texttt{in}_\prec(I)=(y_1y_4)$. If we take all lattice points in the square, the corresponding monomial is $y_1y_2y_3y_4$ and belongs to the initial ideal. Hence, the full set of lattice points of the square does not belong to $\Delta$. Next we focus on the four three-element subsets. Clearly, two of them are not in $\Delta$, namely those corresponding to $y_1y_4y_2$ and $y_1y_4y_3$. However, we obtain two (maximal) faces of $\Delta$ corresponding to $y_1y_2y_3$ and $y_2y_3y_4$. One can check that all other faces of $\Delta$ are subsets of these two maximal ones. We also see that the diagonal $(1,4)$ is the (unique) minimal nonface.
\end{example}
Our next aim is to relate a geometric realization of $\Delta$ to a triangulation of $P$.
\begin{theorem}[{\cite{Stks}, \cite[Theorem 13.25]{ourbook}}]
Using the notation introduced above, $\Delta$ is a triangulation of $P$.  The minimal nonfaces of $\Delta$ correspond to generators of \normalfont{$\rad \texttt{in}_\prec(I)$}.
\end{theorem}
\begin{example}[continues = exm:toricquadric]
We see that $\Delta$ represents one of the possible triangulations of the square. If we change the term order so that $ \texttt{in}_\prec(I)=(y_2y_3)$ we obtain the other triangulation.
\end{example}

Not every triangulation of $P$ may be constructed from term orders in the way described above. Those who do are called \emph{regular}.

\begin{remark} Our definition of regular triangulations may be unfamiliar to readers steeped in combinatorial constructions of regular triangulations. For experts, let us describe the connection. It turns out that for a given ideal, any term order $\prec$ may be induced by associating weights to variables \cite[Proposition 1.11]{Stks}. Hence, the choice of $\prec$ corresponds to assigning weights to the lattice points of $P$. These are precisely the weights used to obtain the regular triangulation.
\end{remark}
In this article we will be mostly interested in unimodular triangulations, i.e.~triangulations into simplices of normalized volume one. Fortunately, initial ideals are very good at detecting those.
\begin{theorem}[{\cite[Corollary 8.9]{Stks}}]
The regular triangulation $\Delta$ is unimodular if and only if \normalfont{$\texttt{in}_\prec(I)$} is a radical ideal.
\end{theorem} 
We next describe a construction, based on \cite{higashitani2019arithmetic}, that produces radical initial ideals for any symmetric edge polytope.
In this case the variables associated to lattice points of $\mathcal{P}_G$ are as follows: one variable $z$, corresponding to the point $\bf{0}$, and for each edge $e\in E$ two variables $x_e$ and $y_e$ corresponding to the two orientations of $e$. We fix a degrevlex order $\prec$ with $z\prec x_{e_1}\prec y_{e_1}\prec\dots\prec x_{e_n}\prec y_{e_n}$ for some ordering of the edge set $E=\{e_1,\dots,e_n\}$.

\begin{theorem}\cite[Proposition 3.8]{higashitani2019arithmetic}\label{thm:GB}
Let $G$ be a simple graph. For any oriented edge $e$, we denote by $p_e$ the corresponding variable, i.e.~$p_e=x_e$ or $p_e=y_e$ depending on the orientation. We also set $q_e$ to be equal to the variable with the opposite orientation, i.e.~$\{p_e,q_e\}=\{x_e,y_e\}$.

The following collection of three types of binomials forms a Gr\"obner basis of the toric ideal $I_{P_G}$ with respect to the order $\prec$:
\begin{enumerate}
\item For every $2k$-cycle $C$ of $G$, with fixed orientation, and any $k$-element subset $J$ of edges of $C$ not containing the smallest edge among those of $C$ in the chosen ordering,
 $$\prod_{e\in J}p_e-\prod_{e\in C\setminus J} q_e.$$
\item For every $(2k+1)$-cycle $C$ of $G$, with fixed orientation, and any $(k+1)$-element subset $J$ of edges of $C$, 
$$\prod_{e\in J}p_e-z\prod_{e\in C\setminus J}q_e.$$
\item For any edge $e$ of $G$, $$x_ey_e-z^2 \, .$$
\end{enumerate}
Note that the leading monomial is always chosen to have positive sign.
\end{theorem}
Observe that the initial ideal from Theorem \ref{thm:GB} is always radical. Further, a monomial $m$ in the variables $x_e,y_e$ belongs to this initial ideal if and only if $zm$ does. Hence, the induced triangulation $\Delta$ is unimodular and a simplex in the boundary belongs to $\Delta$ if and only if its extension by ${\bf 0}$ does. This gives us the following strategy for the computation of the normalized volume when $G$ is connected:
\begin{enumerate}
\item Determine the facets of $\mathcal{P}_G$, by Theorem \ref{thm:facets}.
\item For each facet consider the simplices in that facet, corresponding to special spanning trees as in Corollary \ref{cor:simpinfacet}.
\item Count those simplices/spanning trees whose directed edges, represented as a monomial, are not divisible by any of the leading terms in Theorem \ref{thm:GB}.
\end{enumerate}
\begin{example}[continues = exm:K22]
We have fixed a facet $F$ of a three dimensional polytope $\mathcal{P}_G$ for $G=K_{2,2}$. The graph $G_F$ has four edges going from one side of the bipartite graph to the other as represented in Figure \ref{Fig:K22}. The four edges form a (nonoriented) cycle. Ordering the edges $1\prec 2\prec 3\prec 4$, we obtain $ \texttt{in}_\prec(I)=(y_2y_3)$. Theorem \ref{thm:GB}(1) gives us a Gr\"obner basis element $y_2y_3-y_1y_4$, where the set $J$ equals $\{2,3\}$. Hence, we must count spanning trees that do not contain both of the edges in $J$. Clearly, there are two of them. These correspond to the two triangles $\{1,2,4\}$ and $\{1,3,4\}$ in the square. Hence, the normalized volume of this facet equals two.
\end{example}
\begin{figure}[h!]
\includegraphics[scale=0.2]{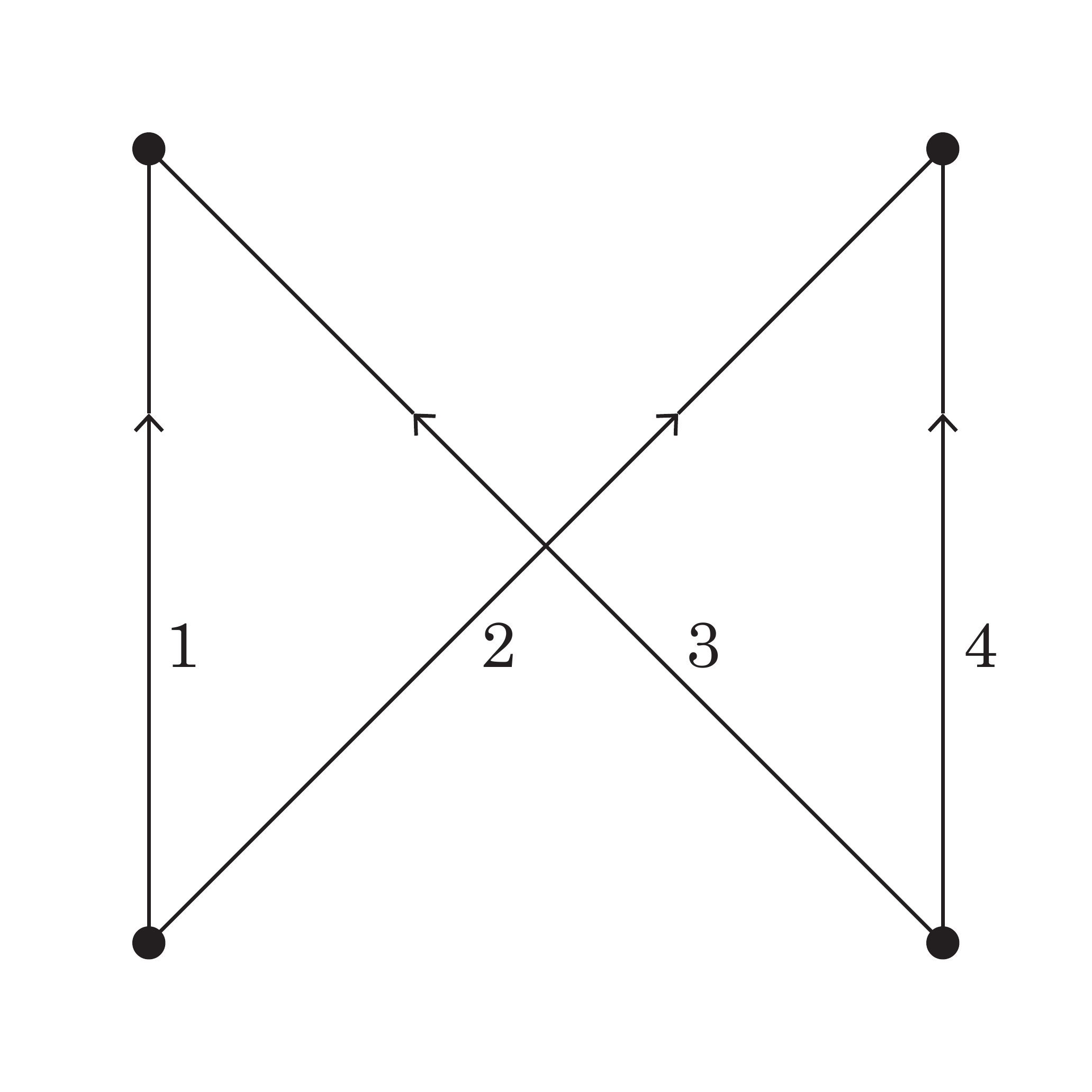}\caption{Orientation of the edges of the graph $K_{2,2}$ corresponding to one facet of $\mathcal{P}_{K_{2,2}}$\,.}\label{Fig:K22}
\end{figure}

\subsection{Ehrhart polynomials and power series}\label{ssec:Ehr}
Let $P$ be a lattice polytope, i.e.~a polytope with vertices in $\ZZ^n$. A celebrated theorem of Ehrhart \cite{ehrhartRational} states that the function $E_P:\ZZ_{\geq 0}\rightarrow \ZZ$ defined by $E_P(j)=|jP\cap\ZZ^n|$ is a polynomial, known as the \emph{Ehrhart polynomial}. A closely related concept is the \emph{Hilbert-Ehrhart series} defined by:
$$\HS_P(t):=\sum_{j=0}^\infty E_P(j)t^j.$$
If $P$ is full-dimensional, this is a rational function of the form $$\HS_P(t)=h^*(t)/(1+t)^{n+1}.$$ Here $h^*(t)=\sum h^{\ast}_jt^j $ is a polynomial of degree at most $n$, known as the \emph{$h^*$-polynomial}. It can be also defined by the formula
\[
E_P (j)=h_0 ^\ast {j+n\choose n}+h_1 ^\ast {j+n-1\choose n} +\cdots + h_n^\ast {j\choose n} \, .
\]
The $h^*$-polynomial encodes many properties of the polytope $P$. For example $P$ is reflexive if and only if $h^*$ is palindromic and of degree $n$, as proved by Hibi in \cite{HibiReflexive}. 

The notion of {\em $f$-vector} can be defined for simplicial complexes, just as in \S\ref{ssec:faces} but setting $f_i$ to be the number of $i$-dimensional simplices. The \emph{$h$-polynomial} of a polytope or a simplicial complex is $h(t):= \sum_{j=0}^{n+1} h_{j}t^j$ where the coefficients are given via the following relation with the $f$-polynomial given in \eqref{eq:fpoly}:
\[f(t)=\sum _{i=0}^{n+1} h_i t^i(1+t)^{n+1-i}. \]
The $h$- and $h^*$-polynomials should not be confused, however they are very much related. The $h$-polynomial of a unimodular triangulation of the lattice polytope $P$ is the $h^*$-polynomial of this polytope. When $P$ is reflexive, this is also the $h$-polynomial of the unimodular triangulation of the boundary. As symmetric edge polytopes have unimodular boundary triangulations, the study of their $h^*$-polynomials is in particular the study of $h$-polynomials of special sphere triangulations. Many intriguing conjectures about triangulations of spheres are still unsolved. We want to mention here the Charney--Davis conjecture and two of its possible strengthenings, due to Gal and Nevo--Petersen: to this end, we need some more definitions.

An $n$-dimensional simplicial complex $\Delta$ is said to be
\begin{itemize}
\item \emph{flag} if its Stanley--Reisner ideal is generated in degree two, i.e.~all minimal nonfaces of $\Delta$ are edges;
\item \emph{balanced} if its vertices can be partitioned into $n+1$ color classes in such a way that no vertices belonging to the same face share the same color.
\end{itemize}
If a simplicial complex $\Delta$ triangulates an $n$-dimensional sphere, it is known that its $h$-polynomial is palindromic of degree $n+1$. As a consequence, it can be expressed as \[h(t)=\sum_{i=0}^{\lfloor\frac{n+1}{2}\rfloor}\gamma_i t^i(1+t)^{n+1-2i},\] where $(\gamma_0, \gamma_1, \ldots, \gamma_{\lfloor\frac{n+1}{2}\rfloor})$ is the \emph{$\gamma$-vector} of $\Delta$. Just like the $h$-polynomial is better suited than the $f$-polynomial to highlight some properties of a simplicial complex, it turns out that the $\gamma$-vector plays a similarly important role when dealing with flag triangulations of spheres.

We can now state the promised conjectures, in increasing order of strength. We remark that all of these are usually stated for a more general class of simplicial complexes, namely flag \emph{generalized homology spheres} (also known as flag \emph{Gorenstein*} complexes).

\begin{con}[{Charney--Davis \cite[Conjecture D, equivalent form on p.~135]{charney1995euler}}] \label{conj:charneydavis}
Let $\Delta$ be a flag triangulation of a $(2d-1)$-dimensional sphere. Then $(-1)^d h(-1) \geq 0$. Equivalently, the last entry $\gamma_d$ of the $\gamma$-vector of $\Delta$ is nonnegative.
\end{con}

\begin{con}[{Gal \cite[Conjecture 2.1.7]{gal2005real}}] \label{conj:gal}
Let $\Delta$ be a flag triangulation of a sphere. Then all the entries of the $\gamma$-vector of $\Delta$ are nonnegative.
\end{con}

\begin{con}[{Nevo--Petersen \cite[Conjecture 6.3]{nevo2011gamma}}] \label{conj:nevopetersen}
Let $\Delta$ be a flag triangulation of a sphere. Then the $\gamma$-vector of $\Delta$ is the $f$-vector of a balanced simplicial complex.
\end{con}

\begin{remark}
At first sight, it is not clear why Conjectures \ref{conj:gal} and \ref{conj:nevopetersen} are stated for spheres of any dimension, in contrast to Conjecture \ref{conj:charneydavis} which deals only with odd-dimensional spheres. Indeed, for any triangulation of a $2d$-dimensional sphere one has that $-1$ is a root of the $h$-polynomial, so the na\"ive version of Conjecture \ref{conj:charneydavis} for even-dimensional spheres is trivially true. However, Gal and Januszkiewicz \cite[Theorem]{gal2010even} note that the general version of the Charney--Davis conjecture is equivalent to asking that, for every $2d$-dimensional flag generalized homology sphere, $(-1)^d\tilde{h}(-1) \geq 0$ holds, where $\tilde{h}(t) := \frac{h(t)}{1+t}$.
\end{remark}

\subsection{Recursions and word counting}\label{ssec:words}
As we will see in Section \ref{ssec:wheel}, some computations about special symmetric edge polytopes will boil down to combinatorial questions about words. In particular, given a finite alphabet and a finite set of forbidden words, we will be interested in counting (cyclic) words that do not contain any of these forbidden words as a subword.

\begin{exm}[label = running_example]
Let us consider the alphabet consisting of three letters $+$, $0$ and $-$. How many words of length $\ell$ avoid the subwords $+-$, $-+$ and $000$? If $\ell=3$, there are 16 such words:
\[\begin{split}&{+}{+}{+}, {+}{+}{0}, {+}{0}{+}, {+}{0}{0}, {\boldsymbol{+}}{\boldsymbol{0}}{\boldsymbol{-}}, {0}{+}{+}, {0}{+}{0}, {0}{0}{+},\\ &{-}{-}{-}, {-}{-}{0}, {-}{0}{-}, {-}{0}{0}, {\boldsymbol{-}}{\boldsymbol{0}}{\boldsymbol{+}}, {0}{-}{-}, {0}{-}{0}, {0}{0}{-}.\end{split}\]
Now imagine that any such word is written on a strip of paper. What happens if we bring together the two endpoints of the strip, so that the last letter of the word is adjacent to the first? Then the number of eligible words drops to 14, since ${+}{0}{-}$ and ${-}{0}{+}$ (in boldface above) are not acceptable anymore.
\end{exm}

Such a problem can be tackled via the \emph{Goulden--Jackson cluster method}, see for instance \cite{GouldenJackson}. We will now describe this procedure (following the exposition by Noonan and Zeilberger \cite{NoonanZeilberger}) and its extension to cyclic words due to Edlin and Zeilberger \cite{EdlinZeilberger}.

\subsubsection{The Goulden--Jackson cluster method}
Fix an alphabet over $k$ letters and a set $B$ of $n$ bad words. We assume without loss of generality that no element of $B$ is a proper subword of another element in $B$. Our goal is to compute the generating function $f(s) := \sum_{i=0}^{+\infty}c_is^i$, where $c_i$ is the number of $i$-letter words that do not contain any element of $B$ as a subword. It turns out that this formal series can be expressed as a rational function in $s$.

Given any $\ell$-letter word $w = w_1w_2 \cdots w_{\ell}$, we define its \emph{weight} to be $\textrm{weight}(w):=s^{\ell}$, its \emph{head} to be the set of all its proper prefixes $\{w_1, w_1w_2, w_1w_2w_3, \ldots, w_1w_2\cdots w_{\ell-1}\}$ and its \emph{tail} to be the set of all its proper suffixes $\{w_2w_3\cdots w_{\ell}, w_3w_4\cdots w_{\ell}, \ldots, w_{\ell}\}$. Given two words $u$ and $v$ (in this order), we define their $\emph{overlap}$ to be the intersection of the tail of $u$ with the head of $v$. Moreover, whenever the overlap of $u$ and $v$ is nonempty, we define \[(u : v) := \sum_{x \in \textrm{overlap}(u, v)}\textrm{weight}\left(\frac{v}{x}\right),\] where $\frac{v}{x}$ denotes the subword of $v$ obtained by erasing the prefix $x$.

For each $v \in B$ we want to compute $L_v$, a certain rational function in $s$. These rational functions are found by solving the linear system in $|B|$ equations and $|B|$ unknowns
\[L_v = {-} \textrm{weight}(v) - \sum_{\substack{u \in B \\ \textrm{overlap}(u, v) \neq \varnothing}}(u : v) \cdot L_u\,.\]

Finally, let $L = \sum_{v \in B}L_v$\,. Then
\begin{theorem}[{\cite{NoonanZeilberger}}]
With the above notation, the rational generating function $f(s)$ satisfies 
$$f(s)=\frac{1}{1-ks-L}\,.$$
\end{theorem}

\begin{exm}[continues = running_example]
Consider the alphabet $\{+, 0, -\}$ and the set of bad words $B = \{+-, -+, 000\}$. This gives rise to the linear system
\[
\begin{cases}
L_{{+}{-}} = {-} \textrm{weight}({+}{-}) - ({-}{+} : {+}{-}) \cdot L_{{-}{+}}\\
L_{{-}{+}} = {-} \textrm{weight}({-}{+}) - ({+}{-} : {-}{+}) \cdot L_{{+}{-}}\\
L_{{0}{0}{0}} = {-} \textrm{weight}({0}{0}{0}) - ({0}{0}{0} : {0}{0}{0}) \cdot L_{{0}{0}{0}}
\end{cases}
\]
which becomes
\[
\begin{cases}
L_{{+}{-}} = {-} s^2 - s \cdot L_{{-}{+}}\\
L_{{-}{+}} = {-} s^2 - s \cdot L_{{+}{-}}\\
L_{{0}{0}{0}} = {-} s^3 - (s+s^2) \cdot L_{{0}{0}{0}}
\end{cases}
\]
and hence gives $L_{{+}{-}} = L_{{-}{+}} = {-}\frac{s^2}{1+s}$ and $L_{{0}{0}{0}} = {-}\frac{s^3}{1+s+s^2}$\,. Thus, \[L = L_{{+}{-}} + L_{{-}{+}} + L_{{0}{0}{0}} = \frac{-2s^2-3s^3-3s^4}{(1+s)(1+s+s^2)}\] and \[f(s) = \frac{1}{1-3s-L} = \frac{-s^3-2s^2-2s-1}{2s^3+2s^2+s-1} = 1+3s+7s^2+\mathbf{16s^3}+36s^4+82s^5+\ldots\] 
As already computed, there are indeed 16 3-letter words satisfying the constraints.
\end{exm}

\subsubsection{The cyclic case}
We now want to address the case when words are \emph{cyclic}, i.e.~the last letter of the word is adjacent to the first (but we still remember where the word starts). As explained in Example \ref{running_example}, this introduces some new constraints that need to be taken care of. Let us introduce some notation in view of this.

Fix an ordering $b_1, \ldots, b_n$ for the bad words in $B$ and define an $n\times n$ matrix $A = (a_{ij})_{i,j=1}^{n}$ by

\[a_{ij} := 
\begin{cases}
{-}(b_i : b_j) & \textrm{if } \textrm{overlap}(b_i, b_j) \neq \varnothing \\
0 & \textrm{otherwise.}
\end{cases}
\]

Define then another $n$ by $n$ matrix $M = (m_{ij})_{i,j=1}^{n}$ by
\[M := A \cdot (I_n-A)^{-1} \cdot s \cdot \frac{dA}{ds},\]
where $I_n$ is the $n$ by $n$ identity matrix and both multiplication by $s$ and differentiation are taken entrywise.

Finally, given any formal power series $\sum_{i=0}^{+\infty}c_i s^i$ and any integer $r > 0$, we set
\[\textrm{chop}_{r}\left(\sum_{i=0}^{+\infty}c_i s^i\right) := \sum_{i=r}^{+\infty}c_i s^i.\]

We are now ready to state the main result from \cite{EdlinZeilberger}.
\begin{theorem}[{\cite[Theorem]{EdlinZeilberger}}]
With the above notation, the rational generating function for the cyclic case equals
\[\frac{1+s\frac{dL}{ds}-L}{1-ks-L} + \sum_{i=1}^{n}\emph{\textrm{chop}}_{\ell_i}(m_{ii}) ,\]
where $\ell_i$ is the length of the $i$-th bad word $b_i$.
\end{theorem}

\begin{exm}[continues = running_example]
One checks that
\[A = 
\bordermatrix{
& {+}{-} & {-}{+} & {0}{0}{0} \cr
{+}{-} & 0 & -s & 0 \cr
{-}{+} & -s & 0 & 0 \cr
{0}{0}{0} & 0 & 0 & -s-s^2
} \quad \textrm{and}
\quad
M = 
\bordermatrix{
& {+}{-} & {-}{+} & {0}{0}{0} \cr
{+}{-} & \frac{s^2}{1-s^2} & \frac{-s^3}{1-s^2} & 0 \cr
{-}{+} & \frac{-s^3}{1-s^2} & \frac{s^2}{1-s^2} & 0 \cr
{0}{0}{0} & 0 & 0 & \frac{s^2+3s^3+2s^4}{1+s+s^2}
}.
\]
Since $\textrm{chop}_2(\frac{s^2}{1-s^2}) = \frac{s^2}{1-s^2}$ and $\textrm{chop}_3(\frac{s^2+3s^3+2s^4}{1+s+s^2}) =  \frac{s^2+3s^3+2s^4}{1+s+s^2} - s^2 =  \frac{2s^3+s^4}{1+s+s^2},$ one has that the generating function is
\[\begin{split}
&\frac{1+s\frac{dL}{ds}-L}{1-3s-L} + \sum_{i=1}^{3}{\textrm{chop}}_{\ell_i}(m_{ii}) = {-}3 + s + s^2 + \frac{-2s^2-6s+4}{2s^4-s^2-2s+1}\\
&= 1 + 3s + 7s^2 + \mathbf{14s^3} + 26s^4 + 62s^5 + 138s^6 + 310s^7 + \ldots
\end{split}\]
As expected, we find 14 3-letter cyclic words satisfying the constraints.
\end{exm}

\settocdepth{subsection}
\section{Results and applications}\label{sec:res}
In this section we show how combining methods from different fields can help us to obtain explicit information about symmetric edge polytopes. Some of the
cases of graphs were already studied, e.g.~trees, cycles, complete bipartite graphs and complete graphs \cite{higashitani2019arithmetic, MHNOH, higashitani2017interlacing, ohsugi2012smooth}. In this section we assume familiarity with the basics of graph theory, and point the reader to Diestel's book \cite{Diestel} as a reference for any undefined terminology.

Our main motivation comes from questions asked by Robert Davis and Tianran Chen:  
``\emph{(...) our current focus is on wheel graphs as well as graphs formed by gluing cycles to cycles or complete graphs to complete graphs.  Any additional information on the normalized volume or facet count will give us important root count information for the Kuramoto model. Description of all the facets will tell us how the subnetworks are formed. Eventually, we will need regular triangulations to construct homotopy algorithms}.''\cite{mailDavis}

On the other hand, Vershik asks to
``\emph{study and classify finite metric spaces according to combinatorial properties of their fundamental polytopes}''\cite[\S{} 1]{vershik2015classification}. Here we focus on  classes of metric spaces that are relevant in the context of computational phylogenetics and whose associated graph can be generated by gluing even cycles and trees (i.e., ``circular split-decomposable metrics'').

\bigskip

We start by presenting the $f$-vector in case of even cycles in Section \ref{ssec:2kcyc}. Cycles played a central role in \cite{chen2018toric}: there, a facet description was particularly important, 
however no general formula for the $f$-vector was given. The symmetric edge polytopes associated with cycles were also studied in depth in \cite{ohsugi2012smooth}. In the case of odd cycles, they were used to disprove two conjectures about the locus of roots of the Ehrhart polynomials of smooth Fano polytopes. 

After this warm-up, in Section \ref{ssec:joins} we show  how the polytopes and their invariants change under various graph-theoretic constructions.
The enumeratively quite challenging treatment of wheel graphs is presented in Section \ref{ssec:wheel}. 
In Section \ref{ssec:splits} we describe fundamental polytopes of circular split-decomposable metrics, as a sample contribution to Vershik's program of combinatorial classification of metric spaces in the special context of metric spaces with relevance in phylogenetics.

Finally, narrowing our focus to bipartite graphs, in Section \ref{sec:flows} we draw a connection with the theory of integer flows on graphs. Using Beck and Zaslavsky's ``inside-out'' approach to Ehrhart theory, we obtain a formula that relates the number of integer points in polar duals of symmetric edge polytopes to the number of facets in the (primal) symmetric edge polytopes of bipartite contractions of the graph. This is a contribution to the ongoing study of the number of integer points in primal-dual pairs of reflexive polytopes \cite{A3,Nill}. In particular, for classes of graphs where one can give explicit expressions for the number of facets (such as those studied in \S \ref{ssec:2kcyc} and \S \ref{ssec:splits}), this formula allows for explicit computation of the number of integer points.

\subsection{Even cycles}\label{ssec:2kcyc}
Let $C_{2k}$ be the even cycle on $2k$ vertices, $k>1$. We will investigate the properties of $\mathcal{P}_{C_{2k}}$. Let us fix a global orientation so that each edge $\{i \pmod{2k}, \ i+1 \pmod{2k}\}$ starts from $i \pmod{2k}$ and ends in $i+1 \pmod{2k}$. Edges oriented according to this orientation will be called \emph{positive} and oriented differently \emph{negative}.
By Theorem \ref{thm:facets} we obtain the following corollary.
\begin{corollary}\label{cor:cyc2k}
The facets of $\mathcal{P}_{C_{2k}}$ are in bijection with integer labelings of the vertices $f\colon V \to \ZZ$ such that:
\begin{itemize}
\item a fixed vertex is labeled by zero;
\item consecutive vertices have labels that differ exactly by one.
\end{itemize}
The polytope $\mathcal{P}_{C_{2k}}$ has $\binom{2k}{k}$ facets. Each facet of $\mathcal{P}_{C_{2k}}$ is $(2k-2)$-dimensional and contains $2k$ vertices, i.e.~(oriented) edges of $C_{2k}$.
\end{corollary}
\begin{proof}
The first statement is exactly the application of Theorem \ref{thm:facets}. To count the facets we proceed as follows. We start by assigning $0$ to the fixed vertex. We follow the cycle, assigning values to vertices, each time either increasing or decreasing the value by exactly one. We obtain a facet precisely when we make exactly $k$ increases and $k$ decreases. In other words, any choice of $k$ edges determines a facet of $\mathcal{P}_{C_{2k}}$ and vice versa. The last statement follows.
\end{proof}

Applying the unimodular triangulation induced by Theorem \ref{thm:GB} one proves that each facet of $\mathcal{P}_{C_{2k}}$ has volume $k$.  

\begin{corollary}[{cf.~\cite[Theorem 2.2]{ohsugi2012smooth}}]
The normalized volume of $\mathcal{P}_{C_{2k}}$ equals $k \cdot \binom{2k}{k}$.
\end{corollary}

Our next goal is to describe the face structure of $\mathcal{P}_{C_{2k}}$. 
\begin{proposition}\label{prop:fullFeven}
The poset of faces $\mathscr F(\mathcal P_{C_{2k}})$ is isomorphic to the set of all ordered pairs $(A,B)$ of disjoint subsets of $[2k]$ where either $\vert A \vert =\vert B \vert = k $ or else $\vert A \vert, \vert B \vert <k$, with partial order given by componentwise containment: $(A,B)\leq (A',B')$ if $A\subseteq A'$ and $B\subseteq B'$.

In particular, the $f$-vector of the polytope $\mathcal{P}_{C_{2k}}$ is given by:
\[f_i(\mathcal{P}_{C_{2k}}) = \begin{cases} \binom{2k}{k} & i = 2k-2\\
\sum_{\substack{a+b=i+1\\a<k, b<k}}\binom{2k}{a}\binom{2k-a}{b} & i < 2k-2\end{cases}.\]
\end{proposition}
\begin{proof} The description of the facets was given in Corollary \ref{cor:cyc2k}. Let us consider lower-dimensional faces.

Fix $i < 2k-2$. Every $i$-dimensional face  $H$ must be contained in some facet $F$ represented by $k$ edges oriented in a positive way and $k$ edges in the opposite direction. The face $H$ must contain at least $i+1$ lattice points corresponding to $a$ edges oriented in a positive way and $b$ in the opposite way, where $a+b=i+1$. If either $a=k$ or $b=k$ then $F$ is the
unique facet containing $H$ and hence $H=F$, which is not possible. Thus we may assume $a,b<k$.
However, then $H$ does not contain any other lattice points. Indeed, consider any additional oriented edge $e$ that corresponds to a point in $F$. To the chosen edges representing lattice points of $H$ we may add one more edge, that is $e$ oriented in the opposite direction, obtaining a set of edges $G'$. Then we may further extend $G'$ to a set of $k$ edges oriented in a positive way and $k$ edges in the negative way. Thus, the initial set of $i+1$ points belongs
to a facet that does not contain $e$. We may repeat the argument for any other edge to see that $H$ is a simplex with $i+1$ points. Hence, for $i<2k-2$, every $i$ dimensional face is represented by a choice of $a<k$ positive and $b<k$ negative edges, where $a+b=i+1$.
\end{proof}

\begin{example}[cf.~Example \ref{exm:K22}]\label{ex:cube}
Let $G=C_4=K_{2,2}$. Then $\mathcal{P}_G$ is a three-dimensional polytope, with six facets that are squares, twelve edges and eight vertices. 
Its face poset is isomorphic to that of a cube, but keep in mind that our polytope is reflexive.  
\end{example}

\subsection{Joining graphs}\label{ssec:joins}
The question we address in this section is what can be said about the symmetric edge polytope of a graph obtained by gluing together two connected graphs. Before getting into this, we record here an easy but useful lemma about symmetric edge polytopes associated with connected bipartite graphs.

\begin{lemma}\label{lem:nonzero}
Let $G$ be a connected bipartite graph. Then, any $f\colon V\to \mathbb Z$ that defines a facet of $\mathcal P_G$ via Theorem \ref{thm:facets} satisfies $\vert f(v)-f(w)\vert = 1$ for all adjacent $v,w$.
\end{lemma}
\begin{proof}
If $f(v)-f(w)=0$ for some adjacent $v, w\in V$, then in the spanning set supporting the nonzero values of $\vert f(v)-f(w)\vert$ there is a path $\pi$ from $v$ to $w$. Since $G$ is bipartite, the length of $\pi$ is odd, but a sum of an odd number of ones and minus ones cannot be zero -- which  it should be in order for the total variation $\vert f(v)-f(w)\vert$ along $\pi$ to be null.
\end{proof}

Lemma \ref{lem:nonzero} gives an immediate upper bound for the number of facets of $\mathcal{P}_G$ when $G$ is bipartite and connected.

\begin{corollary}\label{cor:upperbound}
Let $G$ be a connected bipartite graph. Then $\facets{\mathcal{P}_G} \leq 2^{|V|-1}.$
\end{corollary}

\begin{proof}
Being connected, $G$ contains a spanning tree. After fixing the value of a certain vertex, by repeatedly applying Lemma \ref{lem:nonzero} we have two possibilities for the value of any other vertex in the spanning tree, and hence in $G$ itself.
\end{proof}

\begin{remark}
As follows from the proof, the upper bound in Corollary \ref{cor:upperbound} is realized when $G$ is a tree. In that case, $\mathcal {P}_G$ is combinatorially equivalent to a $(\vert V \vert -1)$-dimensional cross-polytope, cf.~Example \ref{ex:cross}. The bound fails for graphs that are not bipartite: for example, in the case of the complete graph $G=K_n$, the polytope $\mathcal{P}_G$ has $2^n-2$ facets. 
\end{remark}

Let now $G_1$ and $G_2$ be two connected graphs. We first recall what happens when $G_1$ and $G_2$ are joined by one vertex. Let $G$ be a graph obtained by identifying a vertex  $v_1$ of $G_1$ with a vertex $v_2$ of $G_2$. By \cite[Proposition 4.2]{ohsugi2019h} $\mathcal{P}_G=\mathcal{P}_{G_1}\oplus \mathcal{P}_{G_2}$
is the direct sum of the two polytopes and the $h^*$ polynomial for $\mathcal{P}_G$ is the product of the respective $h^*$-polynomials. In particular, the normalized volume is the product of the respective normalized volumes.

\begin{remark}
We note that the polytope $\mathcal{P}_G$ is exactly the same when $G$ is a disjoint union of $G_1$ and $G_2$.
\end{remark}

\begin{remark}\label{rem:poset} Let $P$ and $Q$ be two polytopes with posets of faces $\mathscr F(P)$ and $\mathscr F(Q)$. For any poset $O$ with a unique maximal element let $O^{\setminus \hat 1}$ be the restriction of $O$ to elements that are not maximal. The poset of faces of the direct sum $P\oplus Q$ satisfies 
${\mathscr F}(P\oplus Q)^{\setminus \hat 1} \simeq {\mathscr F(P)}^{\setminus \hat 1}\times {\mathscr F(Q)}^{\setminus \hat 1}$.
\end{remark}

A more sophisticated case is when $G_1$ and $G_2$ are joined by an edge. We start with an easy observation. From now on, let $G$ be the graph obtained from $G_1$ and $G_2$ by identifying an edge.  

\begin{proposition}\label{prop:facetsconnection}
Let $G_1$ and $G_2$ be two connected bipartite graphs. Suppose the symmetric edge polytopes $\mathcal{P}_{G_1}$ and $\mathcal{P}_{G_2}$ have respectively $f_1$ and $f_2$ facets. Then the number of facets of $\mathcal{P}_G$ equals $\frac{1}{2}f_1f_2$.
\end{proposition}
\begin{proof}
Since any symmetric edge polytope $\mathcal{P}$ is centrally symmetric, its facets come in antipodal pairs $(F, -F)$. Let $B(\mathcal{P})$ be the set of such pairs.

We prove the claim by constructing a bijection between $B(\mathcal{P}_{G_1})\times B(\mathcal{P}_{G_2})$ and $B(\mathcal{P}_G)$. Let $e_1=(v_1,w_1)$ (resp.~$e_2=(v_2,w_2)$) be the edge of $G_1$ (resp.~$G_2$) 
that will be identified. Consider a facet $F_1$ of $\mathcal{P}_{G_1}$. By Theorem \ref{thm:facets} we know that such a facet is represented by a function $g_1\colon V_{G_1}\rightarrow\ZZ$, 
where we may assume $g_1(v_1)=0$. Further, as the graph $G_1$ is bipartite, by Lemma \ref{lem:nonzero} we must have $g_1(w_1)\neq 0$. Thus $g_1(w_1)=\pm 1$. 
Analogously for $G_2$ we consider a facet $F_2$ and a function $g_2$ with $g_2(v_2)=0$ and $g_2(w_2)=\pm 1$. By exchanging $g_2$ with $-g_2$ without loss 
of generality we may assume $g_1(w_1)=g_2(w_2)$. As $g_1(v_1)=0=g_2(v_2)$ we see that $g_1$ and $g_2$ induce a function on $G$ that defines a facet. Indeed, on no two vertices joined by an edge the function differs by more than one and the edges on which the function differs exactly by one contain a spanning 
tree. We have thus defined an injective function from $B(\mathcal{P}_{G_1})\times B(\mathcal{P}_{G_2})$ to $B(\mathcal{P}_G)$. Noting that $G$ is bipartite connected itself, one gets the inverse function simply by restricting the defining function $g$ (and $-g$) to $G_1$ and $G_2$.
\end{proof}
\begin{corollary}\label{cor:numfac}
Let $H$ be a graph obtained by joining $k$ even cycles of lengths $2a_1,\dots,2a_k$, consecutively by an edge. The number of facets of $\mathcal{P}_H$ equals
$\frac{1}{2^{k-1}}\prod_{i=1}^k \binom{2a_i}{a_i}$. 
\end{corollary}
\begin{proof}
The proof is by induction on $k$. The case $k=1$ follows from Corollary \ref{cor:cyc2k}. The inductive step is exactly Proposition \ref{prop:facetsconnection}.
\end{proof}
The following proposition may be proved directly. However, we will derive it as an easy corollary of Theorem \ref{thm:hjoin} which describes the $h^*$-polynomial of $\mathcal{P}_G$. It may be also derived from \cite[Corollary 4.5]{ohsugi2019h}.

\begin{proposition}
Let $H$ be a graph obtained by joining $k$ even cycles of lengths $2a_1,\dots,2a_k$, consecutively by an edge. The normalized volume of $\mathcal{P}_H$ equals
$\frac{1}{2^{k-1}}\prod_{i=1}^k a_i \binom{2a_i}{a_i}$.
\end{proposition}

In order to determine the $h^*$-polynomial of $\mathcal{P}_G$ we need a few preparatory lemmas.
\begin{lemma} \label{lemma:initial}
Let $H$ be a bipartite graph and let $e$ be one of its edges. Let $\prec$ be a degrevlex order such that $z \prec x_e \prec y_e \prec v$ for each variable $v \notin \{z, x_e, y_e\}$. Then
\[\texttt{in}_{\prec}I_{\mathcal{P}_H} = (x_ey_e) + J_H,\]
where the generators of the ideal $J_H$ do not involve $z$, $x_e$ nor $y_e$.
\end{lemma}
\begin{proof}
As $H$ has no odd cycles, by Theorem \ref{thm:GB} the generators of the initial ideal are either $x_ey_e$ or do not involve $x_e,y_e$ at all.
\end{proof}
\begin{proposition}\label{prop:initialideals}
Let $G_1$ be a bipartite graph and let $G_2$ be any graph. Let $G$ be the graph obtained from $G_1$ and $G_2$ by identifying one edge $e$. Then there exist degrevlex orders $\prec_1$, $\prec_2$ and $\prec$ such that
\[\texttt{in}_{\prec}I_{\mathcal{P}_G} = \texttt{in}_{\prec_1}I_{\mathcal{P}_{G_1}} + \texttt{in}_{\prec_2}I_{\mathcal{P}_{G_2}},\]
where the rings of $I_{\mathcal{P}_{G_1}}$ and $I_{\mathcal{P}_{G_2}}$ share the variables $z$, $x_e$ and $y_e$. 
\end{proposition}
\begin{proof}
Let $G_1$ have edges $e_1 = e, e_2, \ldots, e_n$ and $G_2$ have edges $e'_1 = e, e'_2,\ldots, e'_m$, where $e_1$ and $e'_1$ are the edges that will be identified in $G$. Let us fix total orders of the variables on $\mathcal{P}_{G_1}$ and $\mathcal{P}_{G_2}$ so that $z \prec x_{e} \prec y_{e} \prec x_{e_2} \prec y_{e_2} \prec \ldots \prec x_{e_n} \prec y_{e_n}$ and $z \prec x_{e} \prec y_{e} \prec x_{e'_2} \prec y_{e'_2} \prec \ldots \prec x_{e'_m} \prec y_{e'_m}$. We let $\prec_1$ and $\prec_2$ be the degrevlex orders with respect to the given orders of the variables. Pick as $\prec$ a degrevlex order with $z \prec x_{e} \prec y_{e} \prec v$ for all variables $v \notin \{z, x_e, y_e\}$ which is also compatible with the variable orders in $G_1$ and $G_2$.
By Theorem \ref{thm:GB} we know that any generator of the squarefree initial ideal of $I_{\mathcal{P}_G}$ is one of the following:
\begin{itemize}
\item $x_{\tilde{e}}y_{\tilde{e}}$ for some edge $\tilde{e}$ in $G$;
\item the product of the variables corresponding to any $k+1$ edges  inside an oriented $(2k+1)$-cycle;
\item the product of the variables corresponding to any $k$ edges inside an oriented $2k$-cycle, provided such edges do not contain the smallest one.
\end{itemize} Let us fix a monomial of the second or third type and let $C$ be the cycle in $G$ from which it arises. If $C$ is entirely contained into $G_1$ (respectively, $G_2$), our monomial appears already in $\texttt{in}_{\prec_1}I_{\mathcal{P}_{G_1}}$ (respectively, $\texttt{in}_{\prec_2}I_{\mathcal{P}_{G_2}}$). If this is not the case, then $C$ does not contain the edge $e$. Note that, since $G_1$ is bipartite, it contains no odd cycle: in particular, $C \cap G_1$ must consist of $2a-1$ edges.
\begin{itemize}
\item If $C$ is a $2k$-cycle, then it consists of $2a-1$ edges in $G_1$ and $2b-1$ edges in $G_2$, where $(2a-1) + (2b-1) = 2k$. Since we need to pick $k = a+b-1$ edges, by the pigeonhole principle we are forced to select at least $a$ edges from $C \cap G_1$ or $b$ edges from $C \cap G_2$. But since $(C \cap G_1) \cup \{e\}$ (respectively, $(C \cap G_2) \cup \{e\}$) is a $2a$-cycle in $G_1$ (respectively, a $2b$-cycle in $G_2$) where $e$ is the smallest edge, the monomial we chose is divisible by a monomial in $\texttt{in}_{\prec_1}I_{\mathcal{P}_{G_1}}$ (respectively, $\texttt{in}_{\prec_2}I_{\mathcal{P}_{G_2}}$), as desired.
\item If $C$ is a $(2k+1)$-cycle, then $C \cap G_2$ consists of $2b$ edges, where $(2a-1) + 2b = 2k+1$. If we pick $a+b$ edges, by the pigeonhole principle we must select at least $a$ edges from $C \cap G_1$ or $b+1$ edges from $C \cap G_2$. The conclusion follows in a similar fashion as in the previous case. 
\end{itemize}
The inclusion 
\[\texttt{in}_{\prec}I_{\mathcal{P}_G} \supseteq \texttt{in}_{\prec_1}I_{\mathcal{P}_{G_1}} + \texttt{in}_{\prec_2}I_{\mathcal{P}_{G_2}},\]
is obvious, which finishes the proof.
\end{proof}
\begin{remark}
The proof of Proposition \ref{prop:initialideals} fails if we join any two graphs, as new elements may appear in the reduced Gr\"obner basis. See also Proposition \ref{prop:join_odd_cycles}.
\end{remark}

\begin{lemma}\label{lem:hilbserfor}
Let $I$ be an ideal in the polynomial ring $\CC[{\bf{x}}]$ and denote by $\HS'$ the Hilbert series of the quotient $\CC[\mathbf{x}]/I$. Denoting by $\HS$ the Hilbert series of $\CC[{\bf{x}},y_1,y_2,z]/\left(I+(y_1y_2)\right)$, the following equality holds:
$$\HS=\frac{1+t}{(1-t)^2}\HS'.$$
\end{lemma}
\begin{proof}
By passing to the initial ideal one can assume that $I$ is a monomial ideal. The Hilbert function of $\CC[{\bf{x}},y_1,y_2,z]/I$ equals $\frac{1}{(1-t)^3}\HS'$, where $\frac{1}{(1-t)^3}=(1+t+t^2+\dots)^3$ counts the exponents of $y_1,y_2$ and $z$ in the monomial. We now have to subtract the monomials divisible by $y_1y_2$ obtaining:
\[\HS=\frac{1}{(1-t)^3}\HS'-\frac{t^2}{(1-t)^3}\HS'=\frac{1+t}{(1-t)^2}\HS'.\]
\end{proof}
The following theorem extends \cite[Corollary 4.5]{ohsugi2019h} from the case of two bipartite graphs to the case of a bipartite graph and an arbitrary graph.
\begin{theorem}\label{thm:hjoin}
Let $G_1$ be a connected graph and $G_2$ be a connected bipartite graph, and denote by $H_1$ and $H_2$ the $h^*$-polynomials of the algebras associated with the respective symmetric edge polytopes. Let $G$ be a graph obtained by joining $G_1$ and $G_2$ by an edge $e_0$.

Then the $h^*$-polynomial for $G$ equals $H_1 H_2/(1+t)$. In particular, this operation preserves real-rootedness of the $h^*$-polynomial.
\end{theorem}
\begin{proof}
Given any graph $G'$, we will denote by $\HS(G')$ the Hilbert series of the algebra associated with $\mathcal{P}_{G'}$, i.e.~$\CC[z,x_e,y_e \mid e\in E(G')] / I_{\mathcal{P}_{G'}}$. If $G'$ is connected, then $\HS(G') = H' / (1-t)^{|V(G')}$, where $H'$ is the $h^*$-polynomial.

Let $\prec$ be a term order as in Lemma \ref{lemma:initial}. Then 
\[\texttt{in}_{\prec}I_{\mathcal{P}_{G_2}} = (x_{e_0}y_{e_0}) + J_{G_2},\]
where the generators of the ideal $J_{G_2}$ do not involve $z$, $x_{e_0}$ nor $y_{e_0}$.

Our goal is to compute $\HS(G)$. To do so, recall that the monomials not in $\texttt{in}_{\prec}I_{\mathcal{P}_G}$ form a basis of $T / I_{\mathcal{P}_G}$, where $T = \CC[z,x_e,y_e \mid e\in E(G)]$. By combining Proposition \ref{prop:initialideals} and Lemma \ref{lemma:initial} we have that \[\texttt{in}_{\prec}I_{\mathcal{P}_G} = \texttt{in}_{\prec}I_{\mathcal{P}_{G_1}}T + J_{G_2}T.\] Since the generators of the monomial ideals $\texttt{in}_{\prec}I_{\mathcal{P}_{G_1}}$ and $J_{G_2}$ do not share any variable, every monomial in $T$ not in $\texttt{in}_{\prec}I_{\mathcal{P}_G}$ may be uniquely represented as the product of a monomial in $\CC[z,x_e,y_e \mid e\in E(G_1)]$ not in $\texttt{in}_{\prec}I_{\mathcal{P}_{G_1}}$ and a monomial in $R := \CC[x_e,y_e \mid e\in E(G_2) \setminus \{e_0\}]$ not in $J_{G_2}$. This implies that $\HS(G) = \HS(G_1)\HS(R / J_{G_2})$. Since by Lemma \ref{lem:hilbserfor} $\HS(R / J_{G_2}) = \HS(G_2)(1-t)^2/(1+t)$, we have that 
\[
\begin{split}
\HS(G) &= \HS(G_1)\HS(G_2)\frac{(1-t)^2}{1+t}\\
&= \frac{H_1H_2}{(1+t)(1-t)^{|V(G_1)| + |V(G_2)| - 2}}\\
&= \frac{H_1H_2}{(1+t)(1-t)^{|V(G)|}}.
\end{split}
\]
Multiplying both sides by $(1-t)^{|V(G)|}$ yields the claim.
\end{proof}

Let us recall that the $\gamma$-vector for a palindromic polynomial $f$ of degree $d$ is defined by the formula $f(x)=\sum_{i=0}^{\lfloor\frac{d}{2}\rfloor}\gamma_i t^i(1+t)^{d-2i}$.
For a graph $G$ we denote by $\gamma_G$ the $\gamma$-vector associated with the $h^*$-polynomial for the toric algebra over the symmetric edge polytope $\mathcal{P}_G$, in the sense of Section \ref{ssec:GB}.
Noting that the $\gamma$-vector remains invariant (up to attaching zeros at the end) after multiplying the polynomial by $(1+t)$, we obtain the following corollary.

\begin{corollary} \label{gamma_concat}
Using the notation of Theorem \ref{thm:hjoin}, the $\gamma$-vector for $G$ equals:
$$({\gamma_G})_i=\sum_{a+b=i}({\gamma_{G_1}})_a ({\gamma_{G_2}})_b.$$
\end{corollary}
As another corollary, we see that the joining of edges provides many examples satisfying the Nevo--Petersen conjecture.
\begin{corollary}\label{cor:nevo}
Let $G_1$ be a connected graph and let $G_2$ be a connected bipartite graph, for which the respective symmetric edge polytopes have unimodular flag triangulations satisfying the Nevo--Petersen conjecture. Let $G$
be the graph obtained from joining $G_1$ and $G_2$ by an edge. Then, any unimodular flag triangulation of the boundary of $\mathcal{P}_G$
satisfies the Nevo--Petersen conjecture.
\end{corollary}
\begin{proof}
By hypothesis there exist $\Delta_1$ and $\Delta_2$ balanced simplicial complexes with $f$-vectors given respectively by $\gamma_{G_1}$ and $\gamma_{G_2}$. Consider the simplicial join $\Delta_1 * \Delta_2$, which is again balanced (since we can pick the coloring induced by those of $\Delta_1$ and $\Delta_2$, assuming the two sets of colors used are disjoint). The $f$-vector of $\Delta_1 * \Delta_2$ is by construction the concatenation of the original $f$-vectors. Applying Corollary \ref{gamma_concat} completes the proof.
\end{proof}
\begin{corollary}
Let $G$ be a graph obtained by successively connecting complete bipartite graphs by an edge. Then the $\gamma$-vector associated to $\mathcal{P}_G$ satisfies the Nevo--Petersen conjecture. 
\end{corollary}
\begin{proof}
The case of one complete bipartite graph is precisely \cite[Theorem 4.2]{higashitani2019arithmetic}. Induction follows by Corollary \ref{cor:nevo}.
\end{proof}
When joining together two non-bipartite graphs by an edge, the situation is different. We include here as a case study the computation of the normalized volume of the symmetric edge polytope associated with the graph obtained by joining two odd cycles by an edge.

\begin{proposition} \label{prop:join_odd_cycles}
Let $G$ be the graph obtained by joining by an edge two odd cycles $C$ and $C'$ of respective lengths $2i+1$ and $2j+1$. Then $\vol(\mathcal{P}_G) = (i+j+2ij)\binom{2i}{i}\binom{2j}{j}$.
\end{proposition}
\begin{proof}
Call $e = \{v_1, v_2\}$ the common edge. Let $F$ be the facet associated with the labeling $f$ as in Theorem \ref{thm:facets}. Without loss of generality, let $f(v_1) = 0$. In what follows we will say that a directed edge $w_1 \to w_2$ is \emph{ascending} (respectively \emph{descending}, \emph{constant}) when $f(w_2) - f(w_1) = 1$ (respectively $-1$, $0$). We choose directions for the edges of $G$ as in Figure \ref{odd_cycles} below.

\begin{figure}[h]
\centering
\includegraphics[trim=0 4cm 0 4cm, scale=0.3]{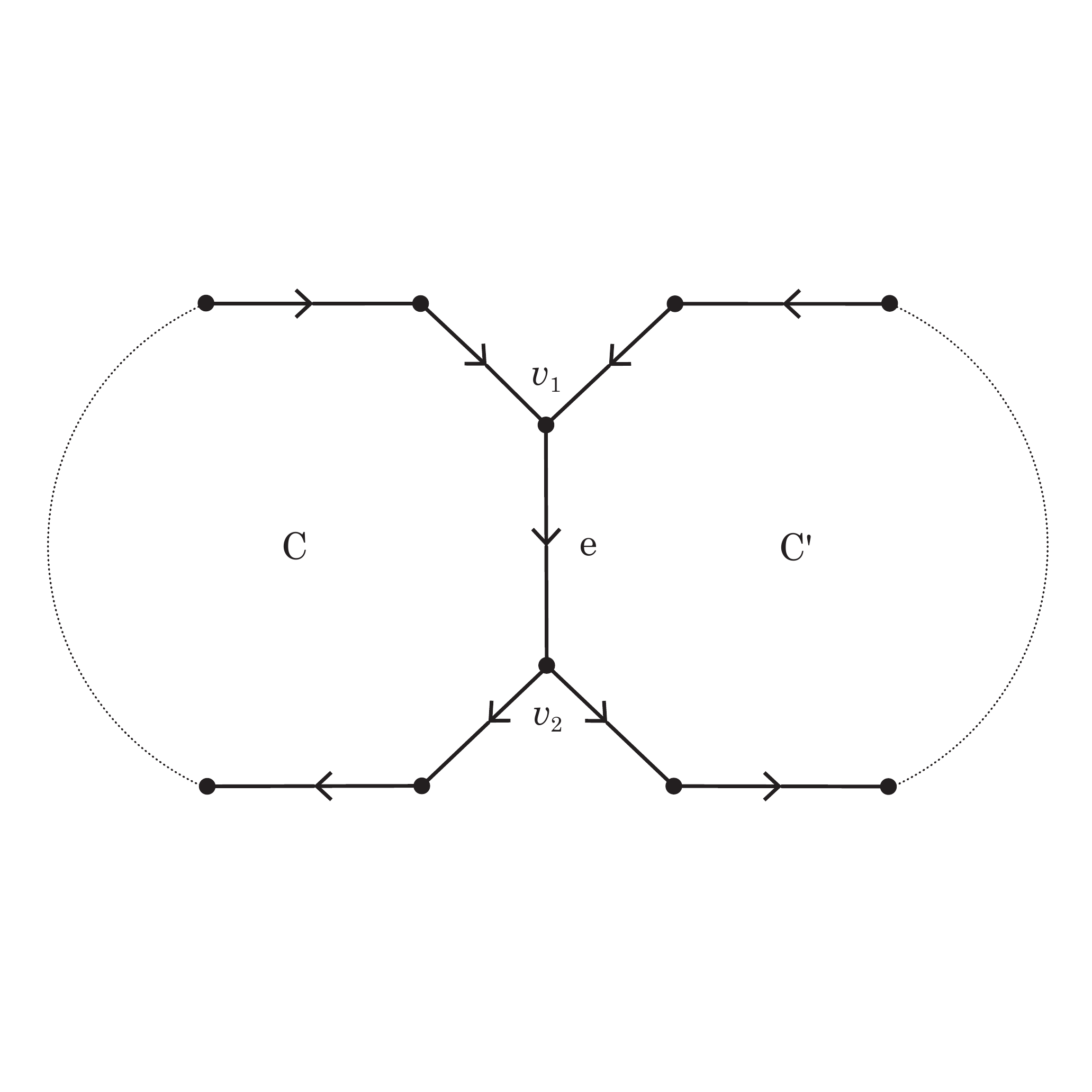}
\caption{Two odd cycles joined by an edge}
\label{odd_cycles}
\end{figure}
 Two possibilities can arise.

\begin{itemize}

\item Assume $f(v_2) = 0$. Then there need to be as many ascending edges as descending ones in $C \setminus \{e\}$ (respectively, $C' \setminus \{e\}$). Moreover, the number of ascending edges must be precisely $i$ (respectively, $j$), as otherwise $G_F$ would either be disconnected or not contain all the vertices. Since the choices on the two cycles are independent, there are $\binom{2i}{i}\binom{2j}{j}$ facets of this type. The simplices contained in such a facet are obtained by taking out a single edge in $C \cup C' \setminus \{e\}$. Since both the edge sets of $C \setminus \{e\}$ and  of $C' \setminus \{e\}$ are equally divided into the two possible orientations, the odd cycles do not impose any condition on the triangulation. However, the existence of the even cycle $C \cup C' \setminus \{e\}$ dictates that we cannot select at the same time all the $i+j$ edges with the opposite orientation as the smallest. This means that the facet $F$ is triangulated into the $i+j$ simplices obtained by taking out one such edge at a time.

\item If $f(v_2) \neq 0$, then $f(v_2) \in \{1, -1\}$. Reasoning in a similar fashion as in the previous case, in order for $G_F$ to contain a spanning tree we need $C$ (respectively, $C'$) to contain $i$ (respectively, $j$) ascending edges, $i$ (respectively, $j$) descending edges and a constant edge $c$ (respectively, $c'$), which cannot be the common edge $e$ by hypothesis. The number of possible facets is then $2i \cdot 2j \cdot 2 \binom{2i-1}{i-1} \binom{2j-1}{j-1}$.
The number $2 \binom{2i-1}{i-1} \binom{2j-1}{j-1}$ is found as follows: after choosing whether $e$ is ascending or descending, we still have to make $i-1$ choices in $C \setminus \{c, e\}$ and $j-1$ independent ones in $C' \setminus \{c', e\}$.
Note that in this case $G_F$ is actually a spanning tree and thus the facet $F$ is a simplex.
\end{itemize}
The total volume is hence \[(i+j)\binom{2i}{i}\binom{2j}{j} + 2 \cdot 2i \cdot 2j \binom{2i-1}{i-1} \binom{2j-1}{j-1} = (i+j+2ij) \binom{2i}{i} \binom{2j}{j}.\]
\end{proof}

\subsection{Wheel graphs}\label{ssec:wheel}
A \emph{wheel graph} is obtained by connecting each vertex of the cycle graph on $n$ vertices (with $n \geq 3$) to a single extra vertex, which we will call cone vertex. Due to conflicting conventions in the literature on whether the graph just described should be called $W_n$ or $W_{n+1}$, we will adopt the notation $K_1 * C_n$.

The goal of this subsection is to compute the number of facets and the normalized volume of the associated symmetric edge polytope.

\begin{proposition} \label{wheel_facet_count}
Denote by $a_n$ the number of facets of $\mathcal{P}_{K_1 * C_n}$. Then the rational generating function for $(a_n)_{n \geq 3}$ is 
\begin{equation} \label{eq:genfunction}
\frac{2z^6+2z^5-7z^4-3z^3+z+1}{(1-z)(1-z-2z^2-2z^3)}.
\end{equation}

In particular, the following recursion holds for $n \geq 3$:
\[a_n = 2a_{n-1} + a_{n-2} - 2a_{n-4}.\]
\end{proposition}
\begin{proof}
We use the notation of Theorem \ref{thm:facets} to describe the facets via appropriate functions $f\colon V(G) \to \ZZ$. Without loss of generality, we assign the value zero to the cone vertex. By condition (i) of Theorem \ref{thm:facets} it then follows that the external vertices can be labeled only by $-1$, $0$ or $1$ (from now on, ``$-$'', ``$0$'' and ``$+$''). We claim that the facets of $\mathcal{P}_{K_1 * C_n}$ are in bijection with the labelings of the external vertices by $\{+,0,-\}$ such that no two consecutive vertices are marked with ``$+-$'' or ``$-+$'' (as this would violate condition (i) in Theorem \ref{thm:facets}) and no three consecutive vertices are marked with ``$000$'' (which would go against condition (ii) in Theorem \ref{thm:facets}). The rational generating function for this counting problem was computed in the third part of Example \ref{running_example}.
\end{proof}
In order to compute the normalized volume we need more work. Let us fix some notation.

When considering facets of the symmetric edge polytope associated with a wheel graph, we will always label by zero the cone vertex. Then, after choosing a vertex on the outer cycle and a direction, each facet is identified by an $n$-letter word $\mathbf{w}$ in the alphabet $\{+, 0, -\}$. More precisely, given such a word $\mathbf{w}$, the associated labeling $f_\mathbf{w}$ gives value $1$ (resp. $-1$) to the outer vertices indexed with ``$+$'' (resp.~ ``$-$'') and value $0$ to both the cone vertex and those outer vertices that are indexed by ``$0$''. If $F_{\mathbf{w}}$ is the facet associated with the labeling $f_\mathbf{w}$, we will denote $G_{F_{\mathbf{w}}}$ by $G_\mathbf{w}$.

Moreover, we will denote by $c(\mathbf{w})$ the number of $4$-cycles in $G_{\mathbf{w}}$ containing the cone vertex.
\begin{proposition} \label{wheel_facet_volume}
Let $F_{\mathbf{w}}$ be the facet of $\mathcal{P}_{K_1 * C_n}$ identified by the word $\mathbf{w}$. Then 
\begin{equation*}\vol(F_{\mathbf{w}}) = 
\begin{cases}
2^{c(\mathbf{w})}-1 & \textrm{if } n \textrm{ is even and } \mathbf{w} \in \mathrm{Exc}_n\\
2^{c(\mathbf{w})} & \textrm{otherwise},
\end{cases}
\end{equation*}
where, if $n$ is even, $\mathrm{Exc}_n$ is the set consisting of the four $n$-letter words ${+}{0}{+}{0}\ldots{+}{0}$, ${0}{+}{0}{+}\ldots{0}{+}$, ${-}{0}{-}{0}\ldots{-}{0}$ and ${0}{-}{0}{-}\ldots{0}{-}$.
\end{proposition}
\begin{proof}
To compute the normalized volume of $F_{\mathbf{w}}$ we will use the strategy outlined at the end of Section \ref{ssec:GB}.

Fixing a total order on the edges of $G$ induces a degrevlex term order as in Theorem \ref{thm:GB}. This gives rise to a regular unimodular triangulation $\Delta$ of $\mathcal{P}_G$. Such a triangulation always features the origin as a cone point, so we are equivalently triangulating (again unimodularly) the boundary $\partial P_G$.

Consider now the Stanley--Reisner ideal $I_{\Delta}$. The algebraic counterpart of restricting $\Delta$ to a facet $F$ consists of adding to $I_{\Delta}$ all the variables corresponding to lattice points not in $F$. Note further that $F$ may contain the lattice point $\mathbf{e}_i - \mathbf{e}_j$ or the lattice point $\mathbf{e}_j - \mathbf{e}_i$, but never both: in algebraic terms, this corresponds to the fact that the monomial $x_ey_e$ (where $e$ is the edge $\{i, j\}$) lives inside $I_{\Delta}$.

In light of all the above, when restricting $\Delta$ to our facet $F_{\mathbf{w}}$, we will consider the Stanley--Reisner ring of the restriction to live in the polynomial ring in the variables corresponding to edges of $G_{\mathbf{w}}$ (for any given edge $e$, we do not need to specify whether to use $x_e$ or $y_e$, as there is only one possible orientation of $e$ inside $G_{\mathbf{w}}$; for simplicity's sake, we will hence just use $e$ to denote the correct variable). With this convention, by Theorem \ref{thm:GB} an edge $e$ appears in the Stanley--Reisner ideal of the triangulation of $F_{\mathbf{w}}$ only if it is contained in some (unoriented) cycle of $G_{\mathbf{w}}$.

We also recall that, if a Stanley--Reisner ideal $I$ decomposes as the sum of Stanley--Reisner ideals $I_1$ and $I_2$ in disjoint sets of variables, then the number of facets of the simplicial complex associated with $I$ is the product of the number of facets of the simplicial complexes associated with $I_1$ and $I_2$.

Now, when $G = K_1 * C_n$, cycles of $G_{\mathbf{w}}$ arise only when we meet a subword of the form $\{\pm, 0, \pm\}$, i.e.~when the three outer vertices are labeled by two nonzero elements on the sides and a zero in the middle. Alternating sequences of the form 
\begin{equation} \label{type_I}
\mathbf{w'} = w'_1\, 0\, w'_2\, 0\, w'_3 \cdots 0\, w'_{k+1}, \ w'_i \in \{+, -\}, \ |\mathbf{w'}| = 2k+1 \leq n
\end{equation}
will give rise to collections of 4-cycles where each cycle shares an edge with the previous one, see Figure \ref{f:type_I} below. If $n$ is even, alternating words of the form
\begin{subequations} \label{type_II}
\begin{align}
\mathbf{w} = w_1\, 0\, w_2\, 0\, \ldots w_k\, 0,\ w_i \in \{+, -\}, \ |\mathbf{w}| = n = 2k \\
\mathbf{w} = 0\, w_1\, 0\, w_2\, \ldots 0\, w_k,\ w_i \in \{+, -\}, \ |\mathbf{w}| = n = 2k
\end{align}
\end{subequations}
will give rise to the $G_{\mathbf{w}}$ drawn in Figures \ref{type_II_gen} and \ref{type_II_exc} below.

It is then enough to show that
\begin{enumerate}[(a)]
\item \label{type_I_case} if we have a subword $\mathbf{w'}$ of type \eqref{type_I}, then the triangulation restricted to the corresponding edges has $2^k$ maximal simplices. 
\item \label{type_II_case} if $n = 2k$ and $\mathbf{w}$ is of type \eqref{type_II}, then $F_{\mathbf{w}}$ is triangulated into $2^k$ simplices unless $\mathbf{w} \in \mathrm{Exc}_n$, in which case we get $2^k - 1$ simplices.
\end{enumerate}
Indeed, for every facet that is not of type \ref{type_II_case} the graph $G_{\mathbf{w}}$ is built from by joining at the cone vertex several graphs of the form given in Figure \ref{f:type_I} (say we have $j$ such parts, with $k_1,\ldots,k_j$ adjacent squares, respectively) and adding some dangling trees, so that $G_{\mathbf{w}}$ has the form illustrated in Figure \ref{fig:glob}. Since all edges of the dangling trees must be in every spanning tree of $G_{\mathbf{w}}$, a choice of a spanning tree for $G_{\mathbf{w}}$ amounts to a choice of a spanning tree in each of the $j$ subgraphs of the form given in Figure \ref{f:type_I}. If Part \ref{type_I_case} holds, then, for all $i=1,\ldots,j$ the $i$-th such subgraph has $2^{k_i}$ spanning trees of the desired type. Therefore, $G_{\mathbf{w}}$ has $2^{k_1}\cdot 2^{k_2}\cdots 2^{k_j}=2^{k_1+\dots+k_j}=2^k$ such spanning trees. This proves that it is in fact enough to prove \ref{type_I_case} and \ref{type_II_case} above.

\begin{figure}[h]
\centering
\includegraphics[trim=0 1cm 0 1cm,scale=0.3]{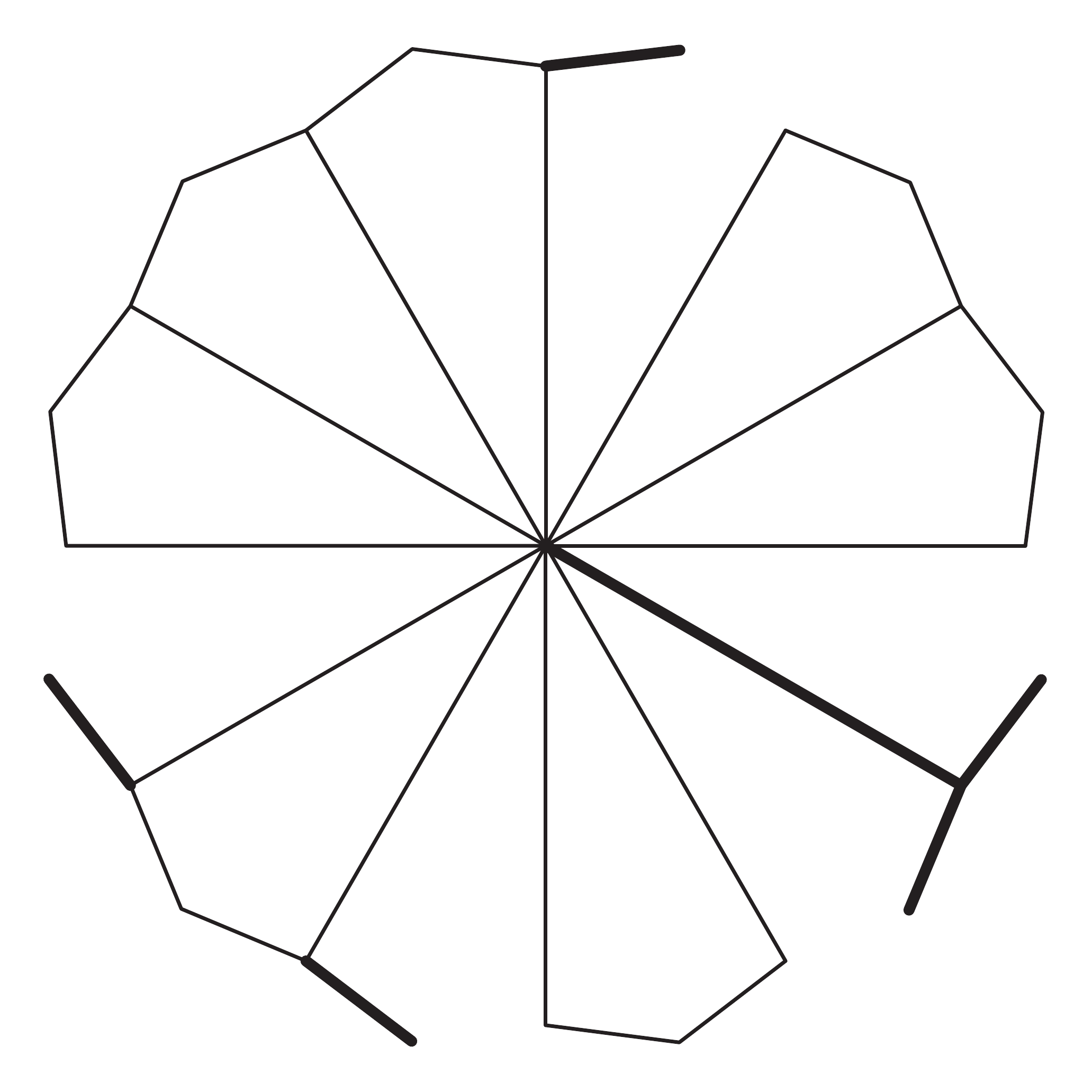}
\caption{Example of a graph $G_{\mathbf{w}}$ without orientation. Edges in bold belong to every spanning subtree.}
\label{fig:glob}
\end{figure}

Let us start with part \ref{type_I_case}. Consider the ordering of the edges of $G_{\mathbf{w'}}$ given in Figure \ref{f:type_I} below, where $e_1$ is the smallest edge, $e_2$ the second smallest and so on. Here and in what follows, the unlabeled edges are bigger than all the labeled ones.

\begin{figure}[h]
\centering
\includegraphics[trim=0 5cm 0 3cm,scale=0.5]{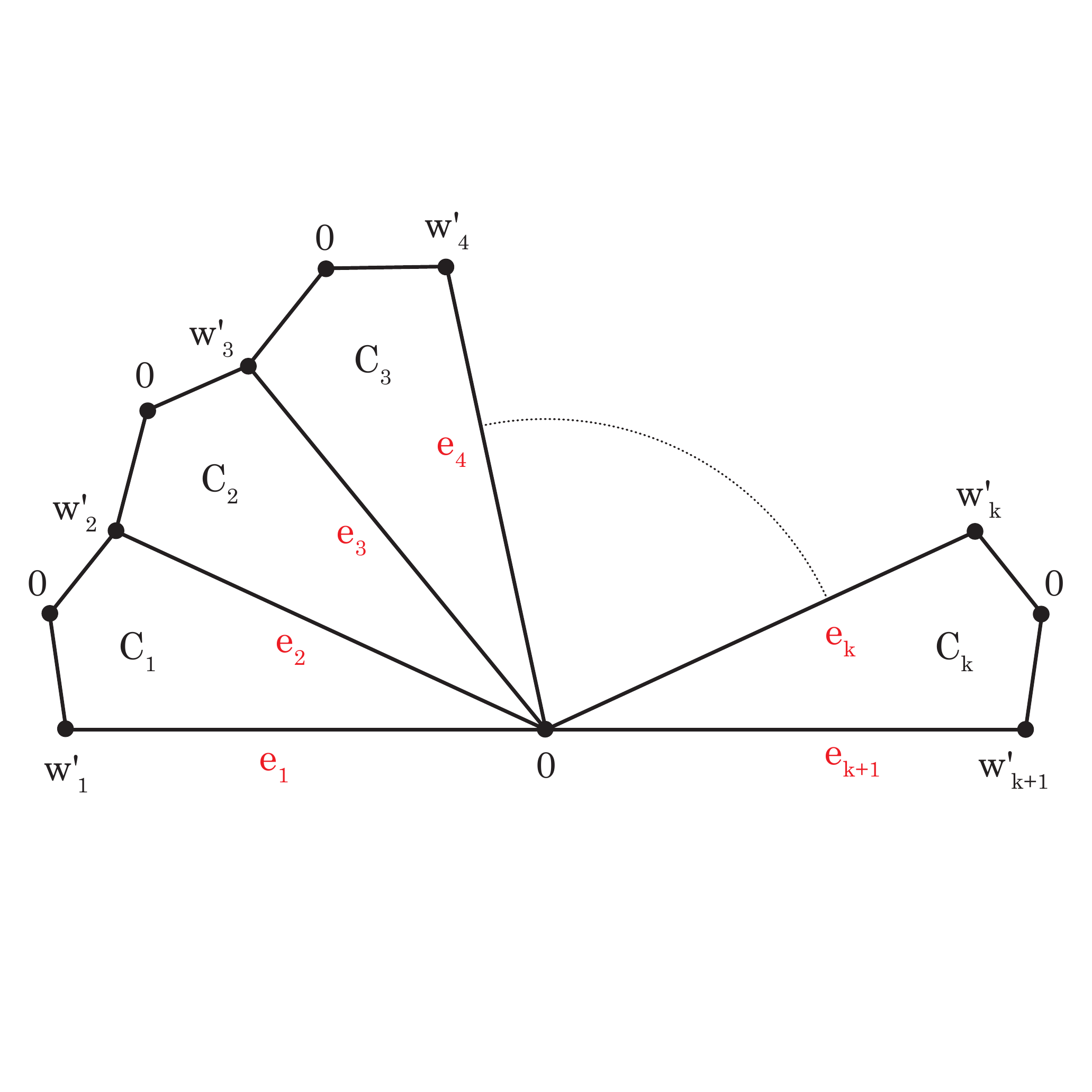}
\caption{The graph $G_{\mathbf{w'}}$ (without orientation), with $\mathbf{w'}$ as in \eqref{type_I}.}
\label{f:type_I}
\end{figure}

Note that each $C_j$ contains two edges oriented clockwise and two oriented counterclockwise. We claim that the Stanley--Reisner ideal associated with $\mathbf{w'}$ is of the form
\[(e_{i_1}e_{i'_1}, e_{i_2}e_{i'_2}, \ldots, e_{i_k}e_{i'_k}),\]
where $e_{i_{\ell}}$ and $e_{i'_{\ell}}$ are distinct elements of $C_{\ell} \setminus \{e_{\ell}\}$ oriented in the same way. In particular, the associated simplicial complex (which is the join of a simplex and a cross-polytope) will have $2^k$ maximal simplices, as desired.
To prove the claim, note that for each cycle $C$ of length greater than four in $G_{\mathbf{w'}}$ there exist distinct indices $j, \ell$ in $\{1, \ldots, k\}$ such that $j < \ell$ and \[C = (C_j \setminus \{e_{j+1}\}) \cup (C_{j+1} \setminus \{e_{j+1}, e_{j+2}\}) \cup \ldots \cup (C_{\ell-1} \setminus \{e_{\ell-1}, e_{\ell}\}) \cup (C_{\ell} \setminus \{e_{\ell}\}).\] Note also that $C$ contains $2(\ell-j+2)$ edges, of which half are oriented clockwise and half counterclockwise; moreover, the smallest edge in $C$ is $e_j$. In particular, the only potential new monomial $\mathbf{m}$ in the Stanley--Reisner ideal is obtained by taking the product of all the edges with the opposite orientation as $e_j$. Now, if $w'_i = w'_j$ for every $i \in \{j+1, j+2, \ldots, \ell+1\}$, then the edge $e_{\ell+1}$ has the opposite orientation as $e_j$ inside $C$. Since the edges in $C_{\ell} \setminus \{e_{\ell}, e_{\ell+1}\}$ are both contained in $C$ and have opposite orientations, we get that $\mathbf{m}$ is divided by $e_{i_{\ell}}e_{i'_{\ell}}$ and is hence superfluous. Otherwise, let $h+1 \in \{j+1, \ldots, \ell+1\}$ be the smallest index such that $w'_{h+1} \neq w'_j$. Then the two edges in $C_h \setminus \{e_h, e_{h+1}\}$ have the same orientation, which is the opposite as $e_j$. In particular, $\mathbf{m}$ is divided by $e_{i_h}e_{i'_h}$.

Let us now prove part \ref{type_II_case}. First consider the case when $\mathbf{w} \notin \mathrm{Exc}_n$. This means that both ``$+$'' and ``$-$'' appear in $\mathbf{w}$. Without loss of generality, we may assume that $w_1 = +$ and $w_k = -$. Order the edges as in Figure \ref{type_II_gen} below.

\begin{figure}[h]
\centering
\includegraphics[scale=0.4]{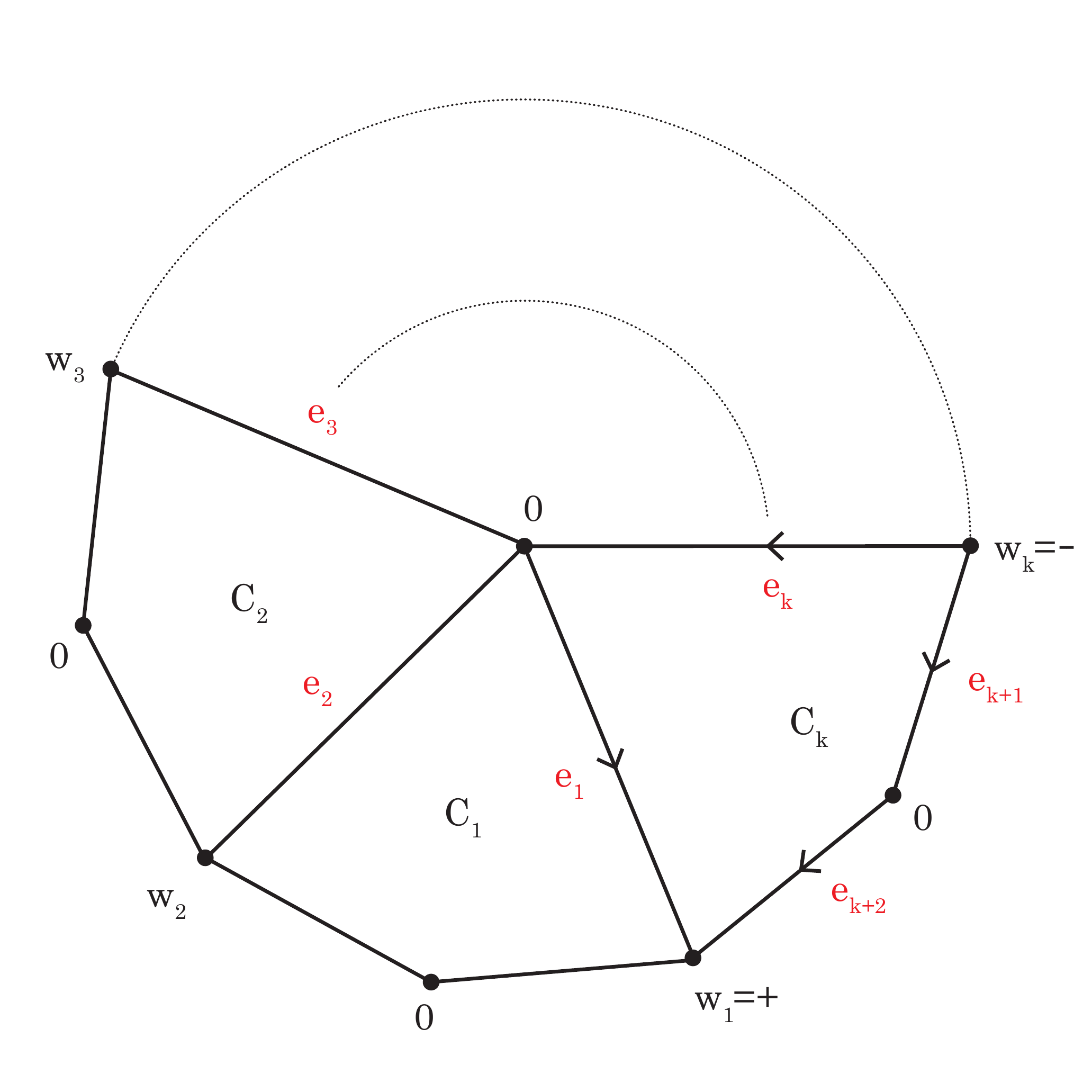}
\caption{The graph $G_{\mathbf{w}}$ with $\mathbf{w}$ as in \eqref{type_II}, $\mathbf{w} \notin \mathrm{Exc}_n$.}
\label{type_II_gen}
\end{figure}

We claim that the Stanley--Reisner ideal associated with $\mathbf{w}$ is of the form
\[(e_{i_1}e_{i'_1}, e_{i_2}e_{i'_2}, \ldots, e_{i_{k-1}}e_{i'_{k-1}}, e_{k+1}e_{k+2}),\]
where $e_{i_{\ell}}$ and $e_{i'_{\ell}}$ are distinct elements of $C_{\ell} \setminus \{e_{\ell}\}$ for $\ell \in \{1, \ldots, k-1\}$ oriented in the same way. Note that $e_{k+1}e_{k+2}$ is in the ideal because, due to our hypotheses, $e_k$ has the same orientation as $e_1$ inside $C_k$. Now let $C$ be a cycle different from $C_1, \ldots, C_k$.
\begin{itemize}
\item If $C$ is the full outer cycle, then it contains $n=2k$ edges equally divided into the two possible orientations. The only potential new monomial $\mathbf{m}$ in the Stanley--Reisner ideal comes from the product of all the edges with the opposite orientation as $e_{k+1}$. Since $e_{k+1}$ and $e_{k+2}$ have the same orientation, by the pigeonhole principle there must be an index $h \in \{1, \ldots, k-1\}$ such that the two edges $e_{i_h}$ and $e_{i'_h}$ of $C_h \setminus \{e_h, e_{h+1}\}$ have the same orientation. Then $e_{i_h}e_{i'_h}$ divides $\mathbf{m}$.
\item If $C$ is not the full outer cycle, then there exist distinct indices $j, \ell$ such that the cycle $C$ starts with $e_j$ and follows clockwise the outer cycle until it ends with $e_{\ell}$. If $\ell > j$, everything works as in case \ref{type_I_case} above. Assume that $\ell < j$. Then $e_{k+1}$ and $e_{k+2}$ lie in $C$ and have the same orientation; moreover, $e_{\ell}$ is the smallest edge in $C$. As before, since the edges in $C$ are equally divided into the two possible orientations, there is at most one monomial $\mathbf{m}$ arising from $C$. If $e_{\ell}$ has the opposite orientation as $e_{k+1}$ and $e_{k+2}$ inside $C$, then $e_{k+1}e_{k+2}$ divides $\mathbf{m}$. If this is not the case, then it must be that $w_{\ell} = -$ (and in particular $\ell > 1$). Then there exists $h \in \{1, \ldots, \ell-1\}$ such that $w_h = +$ and $w_{h+1} = -$. But then the edges $i_h$ and $i'_h$ in $C_h \setminus \{e_h, e_{h+1}\}$ have the same orientation, which is the opposite as $e_j$. Hence, $e_{i_h}e_{i'_h}$ divides $\mathbf{m}$.

\end{itemize}
Finally, let us consider the case when $\mathbf{w} \in \mathrm{Exc}_n$. Without loss of generality, we will assume $\mathbf{w} = {+}{0}{+}{0}\ldots{+}{0}$. In this case we order the edges as shown in Figure \ref{type_II_exc} below.

\begin{figure}[h]
\centering
\includegraphics[scale=0.4]{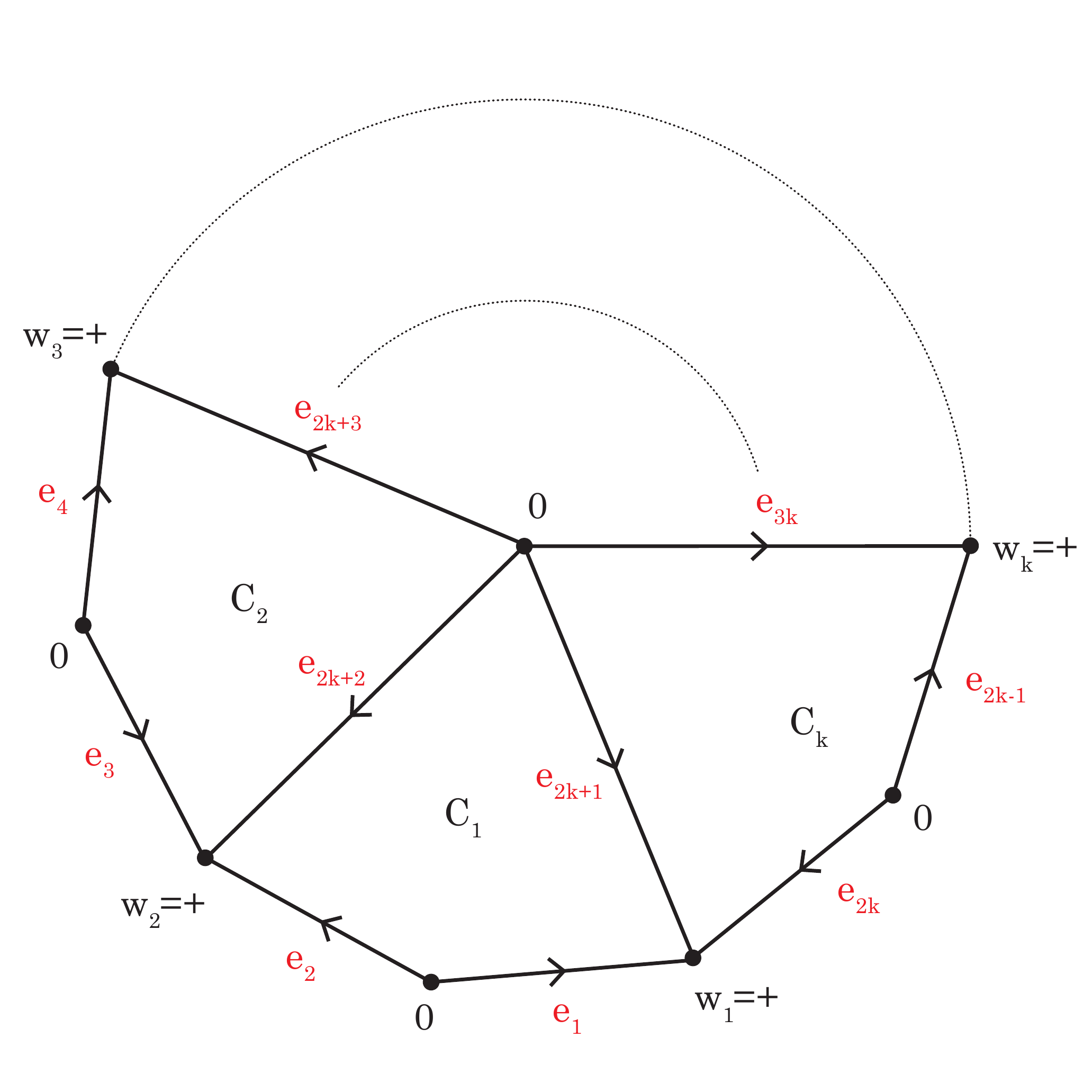}
\caption{The graph $G_{\mathbf{w}}$ for $\mathbf{w} = {+}{0}{+}{0}\ldots{+}{0} \in \mathrm{Exc}_n$.}
\label{type_II_exc}
\end{figure}

We claim that the Stanley--Reisner ideal associated with $\mathbf{w}$ is
\[(e_2e_{2k+1}, e_4e_{2k+2}, e_6e_{2k+3}, \ldots, e_{2k}e_{3k}, e_2e_4e_6\ldots e_{2k}).\]

If this is the case, then the associated simplicial complex has $2^k - 1$ maximal simplices, being the join of a cross-polytope with a single facet removed and a simplex. Let us prove the claim. If $C$ is a cycle different from the outer one and $\mathbf{m}$ is the monomial arising from it, then there exists $h \in \{1, \ldots, k\}$ such that $\{e_{2k+h}, e_{2h-1}, e_{2h}\} \subseteq C$. In particular, $e_{2h}e_{2k+h}$ divides $\mathbf{m}$.
\end{proof}
\begin{theorem}
Let $n \geq 3$ and denote by $\vol_n$ the normalized volume of $\mathcal{P}_{K_1 * C_n}$. Then
\begin{equation*}\vol_n = 
\begin{cases}
(1-\sqrt{3})^n + (1+\sqrt{3})^n & \textrm{if } n \textrm{ is odd}\\
(1-\sqrt{3})^n + (1+\sqrt{3})^n - 2 & \textrm{if } n \textrm{ is even}.
\end{cases}
\end{equation*}
\end{theorem}

\begin{proof}
By Proposition \ref{wheel_facet_volume}, every subsequence of the form $(\pm,0,\pm)$ in a word $\mathbf{w}$ defining a facet contributes a factor $2$ to the overall volume of the facet, unless we are in one of the exceptional cases. Hence, we would like to assign weight two to each appearance of a subword $(\pm, 0, \pm)$. The trick is to enlarge our alphabet to four letters $+, 0, \bar{0}, -$ and count the cyclic words avoiding not just the subwords $+-$, $-+$ and $000$ (as in Proposition \ref{wheel_facet_count}), but also $\bar{0}\bar{0}$, $0\bar{0}$ and $\bar{0}0$. We can then again apply the method of Edlin and Zeilberger to get the rational generating function, that turns out to be
\[\frac{2s^6 + 6s^5-5s^4 - 12s^3 + 2s + 1}{2s^4 + 2s^3-3s^2-2s+1},\] which equals \[-3 + 2s + s^2 + \frac{2s^3-6s^2-6s+4}{2s^4 + 2s^3-3s^2-2s+1}.\]
Since we are interested in the coefficients from $s^3$ onwards, we can ignore the $-3 + 2s + s^2$ part. Since \[\frac{2s^3-6s^2-6s+4}{2s^4 + 2s^3-3s^2-2s+1} = \frac{1}{1-(1-\sqrt{3})s} + \frac{1}{1-(1+\sqrt{3})s} + \frac{1}{1-(-s)} + \frac{1}{1-s},\]
expanding each geometric series we get that the coefficient of $x^n$ is given by
\[
\begin{cases}
(1-\sqrt{3})^n + (1+\sqrt{3})^n & \textrm{if } n \textrm{ is odd}\\
(1-\sqrt{3})^n + (1+\sqrt{3})^n + 2 & \textrm{if } n \textrm{ is even}.
\end{cases}
\]
By Proposition \ref{wheel_facet_volume}, to get the actual volume we need to subtract four from the even case, hence proving the claim.
\end{proof}

\begin{remark}
After sending the first draft, we found out that an alternative approach to computing the normalized volume of the wheel (and its $h^*$-vector) can be taken via \cite[Theorem 4.3]{ohsugi2019h}.
\end{remark}

\subsection{Fundamental polytopes of full planar splits networks}\label{ssec:splits}

As explained in Section \ref{ssec:FundPoly}, the symmetric edge polytope of a simple graph $G$ gives rise, via linear sections, to Kantorovich--Rubinstein polytopes of the metric spaces defined by a certain choice of a subset of the vertices of $G$. The ``full" $\mathcal{P}_G$ is, in fact, the Kantorovich--Rubinstein polytope of the metric space corresponding to choosing the full set of vertices of $G$. In \cite{delucchi2016fundamental} the authors computed the face numbers of Kantorovich--Rubinstein polytopes of all admissible labelings when $G$ is a tree, proving first that all these polytopes are zonotopes and then exploiting a parallel-decomposition of the associated matroid. This setting covers the class of all {\em tree-like} metric spaces, which are a subclass of the so-called {\em split-decomposable} metric spaces \cite{Huson}. In computational phylogenetics, different types of split-decomposable metric spaces are studied, often in terms of the associated {\em splits network}, i.e., a weighted graph that represents the given split-decomposable metric space. 

\begin{definition}[{\cite[\S 5.5]{Huson}}] A {\em splits graph} is represented by a finite, simple, connected, bipartite graph together with an isometric coloring of its edges (i.e., every edge-length-minimal path uses at most once every color and any two minimal paths with the same endpoints use the same set of colors). 
\end{definition}

Given a subset $X$ of the vertices of a splits graph, every color defines a bipartition of $X$ as follows (see \cite[Theorem 5.5.2]{Huson}): the two parts are the subsets of $X$ on either connected component of the graph obtained by removing all edges of the given color.

If we associate a positive real weight $w_c$ to every color $c$, given any two vertices $x,y\in X$ we can associate to every path $\pi$ from $x$ to $y$ the sum $W(\pi)$ of all $w_c$ where $c$ ranges over all colors of edges in the path $\pi$. Then,  we obtain a metric on $X$ by setting the distance of any two $x,y\in X$ to be the minimum of all $W(\pi)$, where $\pi$ ranges over all paths from $x$ to $y$. Every split-decomposable finite metric space can be represented in this way.

\begin{definition}

We call a split-decomposable metric {\em elemental}, resp.\ {\em full} if it can be represented as above on a splits graph by taking every color to have weight $1$, resp.\ by labeling every vertex of the graph, i.e.~taking $X$ to be the set of all vertices. For full, elemental split-decomposable metrics we have $\mathcal K_{G,V}=\mathcal P_G$. 

\end{definition}

\begin{definition}[{\cite[\S 5.7]{Huson}}]
A split-decomposable metric is called {\em circular} if it can be represented on a splits network that admits a planar drawing where all labeled nodes are in the boundary of the unbounded region.
\end{definition}

Recall that an outerplanar graph is a planar graph that has a drawing where every vertex is in the boundary of the unbounded face. Then, a full circular split-decomposable metric is one that can be represented by a fully labeled outerplanar splits network. If in addition such a metric is elemental, we can study the associated Kantorovich--Rubinstein polytope by looking at the symmetric edge polytope of the network graph.

\begin{proposition}\label{prop:BCO}
Let $G$ be a bipartite, connected, outerplanar graph. 
Then, the number of facets of $\mathcal {P}_G$ and its normalized volume are
$$
\facets{\mathcal P_G}=
2^{t-s}\prod_{i=1}^{k}{2a_i \choose a_i}, 
\quad\quad\quad
\operatorname{vol}(\mathcal P_G)
=2^{t-s}\prod_{i=1}^{k}a_i{2a_i \choose a_i}, 
$$
where $a_1,\ldots, a_k$ are the half-lengths of the boundaries of the bounded regions, $s$ is the number of edges separating two bounded regions and $t$ is the number of bridges of $G$.
\end{proposition}

\begin{example}
In the graph of Figure \ref{fig:triple}.(a) we have $k=5$ with $a_1=a_2=a_3=a_4=2$, $a_5=3$, $s=3$ and $t=3$. Thus, the symmetric edge polytope of this graph has $6^420=25920$ facets and a normalized volume of $1244160$. 
\end{example}

\begin{proof}
The blocks \cite[\S 3.1]{Diestel} of an outerplanar graph are either single edges or biconnected (outerplanar) graphs. The block graph (i.e., the intersection graph of the blocks) is a tree  \cite[Lemma 3.1.4]{Diestel}, so we can enumerate the blocks $B_1,\ldots,B_n$ according to a ``reverse pruning order'' of this tree, i.e., an ordering of the vertices such that the vertex-induced subgraph on the first $i$ vertices is connected, for every $i$ (see Figure \ref{fig:triple}.(b)).

\begin{figure}
\includegraphics[scale=0.25]{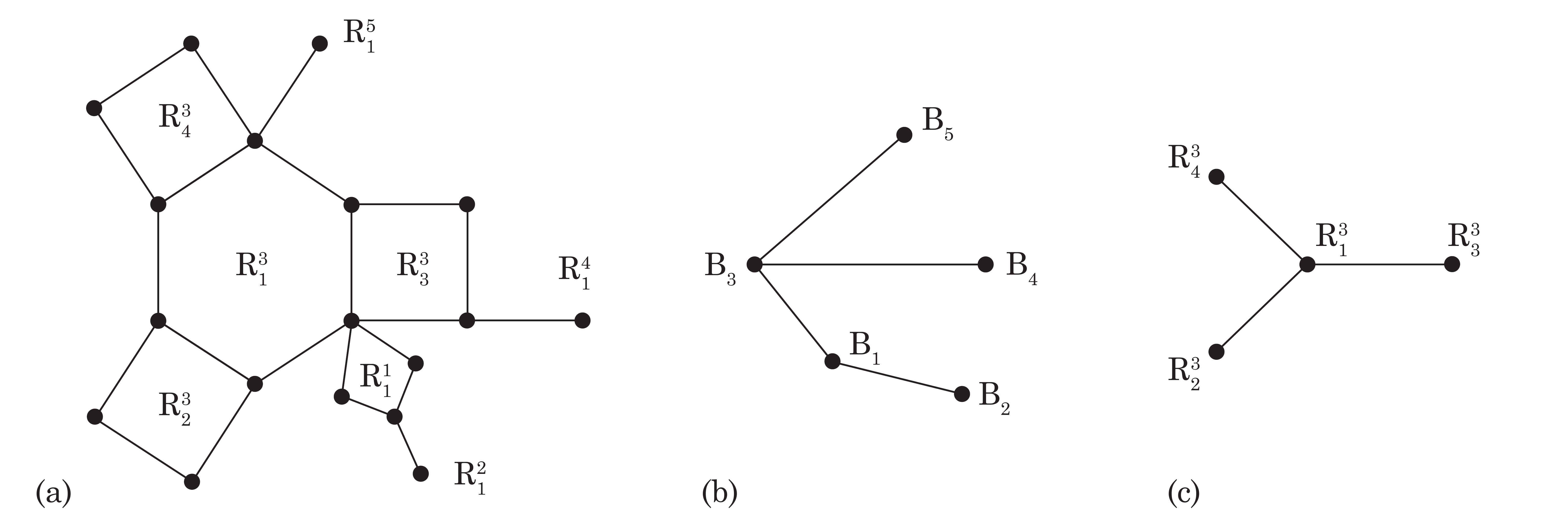}
\caption{(a) A bipartite outerplanar graph. (b) Its block graph. (c) The cycle graph of the block $B_3$.}
\label{fig:triple}
\end{figure}

Moreover, by \cite{Syslo}, the cycle graph of every biconnected block $B_i$ with respect to the cycle basis given by the bounded faces is a tree, and hence we can enumerate the boundary cycles of bounded regions of $B_i$, say $R_1^i,\ldots,R_{j_i}^i$, again according to a reverse pruning order. Notice that, without loss of generality, we can assume that the cycle $R_1^i$ has nonempty intersection with the union of the blocks $B_1,\ldots,B_{i-1}$ (every vertex of a tree can be chosen as the start of a reverse pruning order). If $B_i$ is not biconnected, i.e., it is a single edge, we let $R_1^i=B_i$ and $j_i=1$ (see Figure \ref{fig:triple}.(c)). The ordering
\begin{equation}\label{eq:seq}
R_1^1,\ldots,R_{j_1}^1,R_1^{2},\ldots,R_{j_2}^2,\ldots, R_{j_n}^n
\end{equation}
exhibits $G$ as a sequence of elementary joins. Notice that every $R^i_h$, $h>1$ is an even cycle -- call $2a$ its length -- joined along an edge to the part of the graph constructed earlier; thus it contributes a factor $\frac{1}{2}{2a \choose a }$ to the number of facets and a factor $\frac{1}{2} a{2a \choose a }$ to the normalized volume of the symmetric edge polytope. On the other hand, every $R^i_1$, $i>1$ is joined to the previous part by identifying a vertex. Thus, if $R^i_1$ is an even cycle of length $2a$ it will contribute a factor ${2a \choose a }$ to the number of facets and a factor $a{2a \choose a }$ to the normalized volume. Otherwise, if $R^i_1$ is a single edge it will contribute a factor $2$ to the facets and $2$ to the normalized volume. In the sequence \eqref{eq:seq} every boundary of a bounded region and every bridge of $G$ appears exactly once, and the number of attachments along an edge is exactly $s$.
\end{proof}

We obtain immediately the following corollary.

\begin{corollary}
Let $(X,d)$ be a full, elemental and circularly split-decomposable metric space. Then the number of facets and the normalized volume of the associated Kantorovich--Rubinstein polytope can be computed as in Proposition \ref{prop:BCO} from any drawing of a  (outerplanar, full) splits network representing $(X,d)$.

Moreover, if all biconnected blocks of the splits network are cycles, then the polytope is the direct sum of the polytopes of the cycles and the bridges (see Proposition \ref{prop:fullFeven}).
\end{corollary}

\begin{example}
Consider the full splits network of Figure \ref{SplitNC}. It is one of the basic examples of non-compatible split metric spaces. The associated fundamental polytope has the combinatorial type of the direct sum of a $4$-dimensional crosspolytope (i.e., the direct sum of the $4$ edges) and a cube (i.e., the polytope of the $4$-cycle, see Example \ref{ex:cube}). It has therefore $16\cdot 6=96$ facets and normalized volume equal to $16 \cdot 12 = 192$.

\begin{figure}[t]
\centering
\includegraphics[scale=0.3]{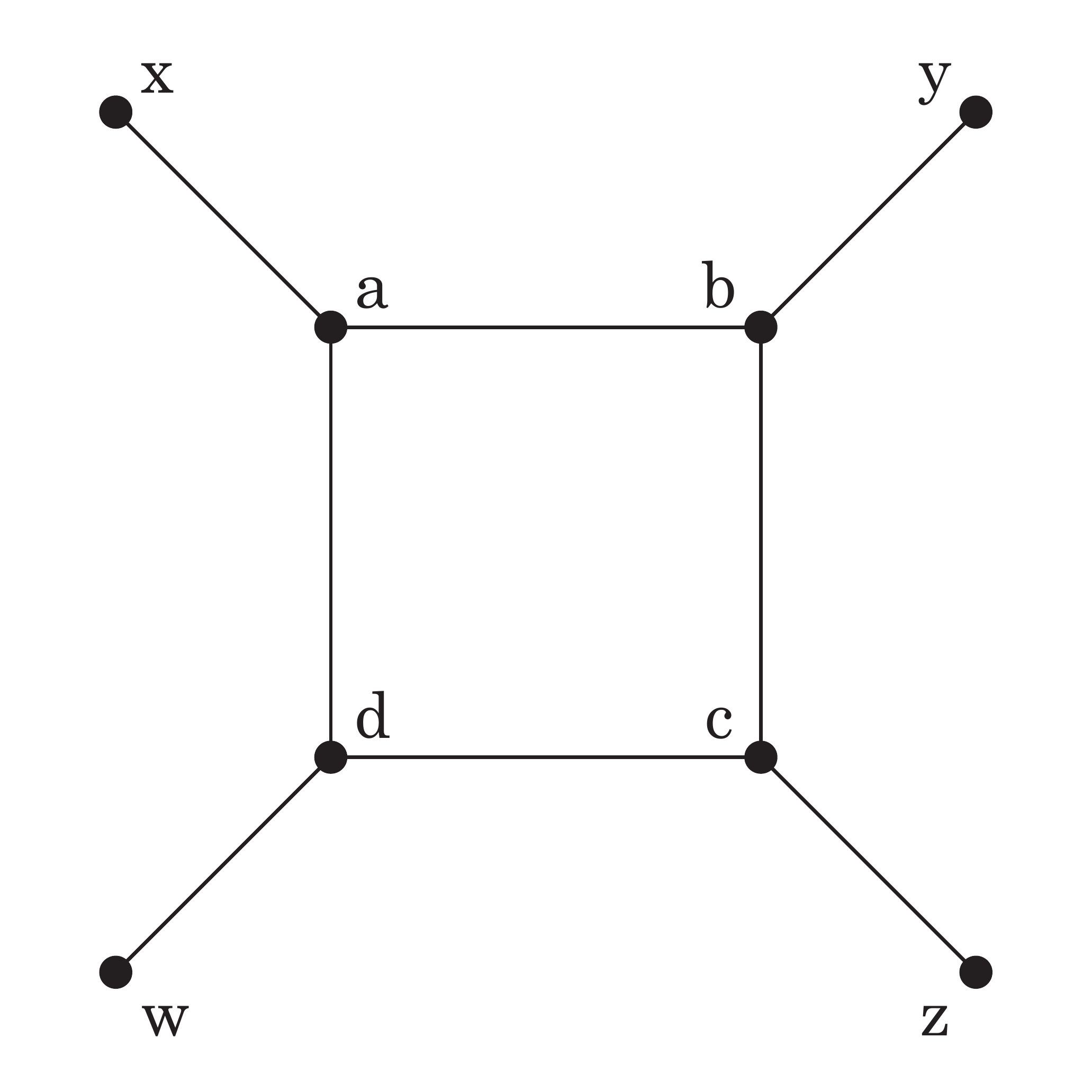}
\caption{}
\label{SplitNC}
\end{figure}
\end{example}

\subsection{Bipartite planar graphs, flows and polar duals of symmetric edge polytopes}\label{sec:flows}
In this section we focus on bipartite planar graphs and connect the combinatorics of their symmetric edge polytopes to the theory of integral flows. This allows us to derive a general bound on the number of faces and a result on polar duals of symmetric edge polytopes.

As is usual when talking about flows in graphs, we will need to arbitrarily specify a {\em tail} and a {\em head} of every edge of a graph $G=(V,E)$. Formally, we will consider two functions $h,t: E\to V$ such that $\{t(e),h(e)\}$ is the set of vertices incident to $e$, for every $e\in E$.

A {\em flow} on $G$ with values in an Abelian group $\mathbb A$ is any $x\in \mathbb A^E$ that satisfies the conservation condition: $\sum_{e\in h^{-1}(v)}x_e = \sum_{e\in t^{-1}(v)} x_e$. Such a flow is called {\em nowhere zero} if no component of $x$ is the identity of $\mathbb A$. Given $k\in \mathbb Z$, a $\mathbb Z_k$-flow is any flow with values in the cyclic group $\mathbb Z_k=\{0,\ldots,k-1\}$. A $k$-flow is a flow with values in the group $\mathbb Z$ and such that $\vert x_e \vert <k$ for all $e\in E$. 

It is classically known that the number of nowhere-zero $\mathbb Z_k$-flows is expressed by a polynomial in $k$, the {\em flow polynomial} $\mfl_G(k)$ of $G$. In particular, when $G$ is a planar graph the flow polynomial is related to the chromatic polynomial of the dual graph $G^*$: $\mfl_G(k)=k^{-1}\chi_{G^*}(k)$. 

A result by Kochol \cite{Kochol} states that also the number of nowhere-zero $k$-flows is expressed by a polynomial, which we will denote by $\ifl_G(k)$.

\begin{proposition}\label{prop:facetsbipartite}
Let $G$ be a bipartite, planar, connected graph. Then the number of facets of $\mathcal{P}_G$ is the number $\ifl_{G^*}(2)$ of nowhere-zero $2$-flows on the dual graph $G^*$.
\end{proposition}
\begin{proof}

Facets of $\mathcal{P}_G$ correspond bijectively to edge-labelings $\lambda: E\to \{\pm 1, 0\}$ that are nonzero on some connected spanning set of edges and that "sum to $0$ on every oriented circuit of $G$" (formally: $\sum_{i}\epsilon_i\lambda(e_i)=0$ whenever $v_0,e_1,v_1,e_2,\ldots,e_l$ is a circuit of $G$ -- i.e., $\{v_i,v_{i+1}\}=\{t(e_i),h(e_i)\}$ --, with $\epsilon_i=1$ if $v_i=t(e_i)$ and $\epsilon_i=-1$ otherwise). In fact, functions $f$ satisfying the conditions in Theorem \ref{thm:facets} define labelings $\lambda$ as above by setting 
$\lambda(e):=f(h(e))-f(t(e))$ for all $e\in E$, and this correspondence is one-to-one. Now, if $G$ is bipartite, such $\lambda$ can never assume the value zero by Lemma \ref{lem:nonzero}.

If $G$ is planar, then the edges of $G$ correspond bijectively to the edges of the dual graph $G^*=(V^*,E^*)$: call $e\mapsto e^*$ this bijection. Now any labeling $\lambda$ of the edges of $G$ induces a labeling $\lambda^*$ of the edges of $G^*$ via $\lambda^*(e^*):=\lambda(e)$. Recall also, e.g., from \cite[Lemma 6.5.2]{Diestel} that there is a (canonical) choice of $h(e^*)$ and $t(e^*)$ such that $\lambda^*$ is a flow if and only if $\lambda$ ``sums to $0$ on every circuit of $G$'' in the sense above.

If $G$ is planar and bipartite, then, valid facet-defining labelings $\lambda$ correspond bijectively to nowhere-zero $2$-flows $\lambda^*\in \{\pm 1\}^E$ on $G^*$.
\end{proof}

\begin{remark}
Via \cite[Remark 1]{Kochol} we immediately recover that when joining two bipartite (planar, connected) graphs by at most one vertex the number of facets of the symmetric edge polytope of the resulting graph is the product of the numbers of facets of the two joined graphs (cf. Remark \ref{rem:poset}).
\end{remark} 

If $G=(V,E)$ is a graph, let $A_G$ denote the signed incidence matrix of $G$. This is a unimodular $\vert V \vert \times \vert E \vert$ matrix whose entry in row $v$ and column $e$ is $1$ if $h(e)=v$, $-1$ if $t(e)=v$ and $0$ otherwise. We let $E_{G}(t)$ denote the Ehrhart polynomial of the polytope
$$
\mathcal Q_G:=[-1,1]^{E} \cap \ker (A_G) \subseteq \mathbb R^{E}.
$$

\begin{lemma} \label{lem:EE}
Let $G$ be a planar graph. Then,
$$E_G(t)=E_{\mathcal P_{G^*}^{\triangle}}(t),$$
where $^{\triangle}$ denotes polar duality of polytopes.
\end{lemma}
\begin{proof}
Consider the signed incidence matrix $A_{G^*}$ of the (planar) dual of $G$ with, e.g., the canonical orientation \cite[Lemma 6.5.2]{Diestel}. It is well known that $\ker (A_G)$ equals the rowspace of $A_{G^*}$. Let us call $W$ this linear subspace of $\mathbb R^E$. Writing $w_0,\ldots, w_s$ for the rows of $A_{G^*}$, we may suppose that $w_1,\ldots,w_s$  correspond to vertices of $G^*$ associated to bounded faces of $G$. Then, $w_1,\ldots,w_s$ give a unimodular basis of $W$ \cite[\S 14.7]{GodsilRoyle}: this means that 
$$W \cap \mathbb Z^d = \langle w_1,\ldots,w_s \rangle_{\mathbb Z}.$$

In particular, any $x\in \ker (A_G)$ has a unique expansion
\begin{equation}\label{eq:x}
x=\sum_{i=1}^s r_iw_i,\quad\quad r_i\in \mathbb R.
\end{equation}
As is customary, we label the columns of $A_{G^*}$ by the edge set $E$ of $G$ in the natural way, and we write $A_{G^*}(e,i)$ for the entry in the $e$-th column and $i$-th row. Let 
$$Z:=\{i\in [s] \mid \textrm{ for some }e\in E:\,\, A_{G^*}(e,j)= 0 \textrm{ iff } j\in [s]\setminus i\},$$ 
i.e., the indices of all vertices of $G^*$ that are adjacent to the vertex associated to the unbounded face of $G$. Now, for every $e\in E$ the $e$-coordinate of the point $x$ as in \eqref{eq:x} is
$$
x_e=\left\{
\begin{array}{ll}
A_{G^*}(e,i)r_i & \textrm{if } e \textrm{ is incident to }w_0,w_i\textrm{ in }G^* \\
A_{G^*}(e,i)(r_i - r_j) & \textrm{if } e \textrm{ is incident to }w_i,w_j\textrm{ in }G^*, i,j\neq 0 \\
\end{array}
\right.
$$
Since the nonzero entries of $A_{G^*}$ are $+1$ or $-1$, we see that
$$
\mathcal Q_G=\left\{x\in \ker (A_G) 
\left\vert
\begin{array}{l}
 \vert r_i \vert \leq 1 \textrm{ for }i\in Z,\\
\vert r_i - r_j \vert \leq 1 \textrm{ for every }e\in E \textrm{ with }\{h(e),t(e)\}=\{i,j\}\subseteq [s]
\end{array}
\right. 
\right\}.
$$
Now consider the subspace $U$ defined in $\mathbb Z^{(V^*)}$ as the set of all points whose coordinates sum to zero. Any spanning tree $T$ of $G^*$ gives a unimodular basis $\{u_e\}_{e\in E(T)}$ of $U$, and hence of $U^*$, by $u_e=\epsilon_{i}-\epsilon_j$ if $e=\{i,j\}$.  The linear transformation $\Lambda:U\to W$ defined by setting, for all $e\in E(T)$,
$$
u_e\mapsto\left\{
\begin{array}{ll}
w_i & \textrm{if } e \textrm{ is incident to }w_0,w_i\textrm{ in }G^* \\
w_i-w_j & \textrm{if } e \textrm{ is incident to }w_i,w_j\textrm{ in }G^*, i,j\neq 0 \\
\end{array}
\right.
$$
maps the unimodular basis $\{u_e\}_{e\in E(T)}$ of $U$ into the unimodular basis $\{w_i\}_{i\in [s]}$ of $W$ and has $\Lambda(\mathcal P_{G^*}^{\triangle})=\mathcal Q_G$. The claim follows.
\end{proof}

The ``inside-out'' approach to Ehrhart theory by Beck and Zaslavsky {\cite{BZ1}} leads to an explicit expression for the integer flow polynomial  which, in our situation, becomes the following.

\begin{proposition}\label{prop:inout} 
Let $G$ be a planar bipartite graph. Then,
$$\facets{\mathcal P_G}
=\sum_{H\in \mathcal L(G)} \mu_{\mathcal L(G)}(H) 
E_{\mathcal P^\triangle _{{G/H}}}(1)
$$
where $\mu_{\mathcal L(G)}$ is the M\"obius function of $\mathcal L(G)$, the poset of all subsets $H\subseteq E$ that are closed in the graphic matroid of $G$, and where we identify the set of edges of $G$ and of its planar dual $G^*$.
\end{proposition}
\begin{proof}
Beck and Zaslavsky prove, for every graph $G$ and every $k$, the identity\footnote{We point out a typo in \cite[Theorem 4.15]{BZ1}: on the right-hand side of Formulas (4.3) and (4.4) the polynomial should be evaluated at $k$ instead of $k+1$.}
\begin{equation}\label{eq:BZ}
\varphi_{G}(k) = \sum_{T\in\mathcal L^*(G)}\mu(\hat{0},T)E_{{G[T^c]}}(k-1)
\end{equation}
where $\mathcal L^*(G)$ is the lattice of flats of the dual matroid to the cycle matroid of $G$ -- which is isomorphic to the lattice of flats $\mathcal L(G^*)$ of the cycle matroid of $G^*$ --  and $\mu$ is the associated M\"obius function. The claim follows by duality with Proposition \ref{prop:facetsbipartite} and Lemma \ref{lem:EE}.
\end{proof}

\begin{proposition} Let $G$ be a planar connected graph. Then, the number of integer points contained in the polar dual of the symmetric edge polytope of $G$ is 
$$
E_{\mathcal P_{G}^\triangle}(1)
=\sum_{
\substack{
S\in \mathcal L(G)\\
G/S \textrm{ bipartite}
}
}
\facets{\mathcal P_{G/S}}.
$$
\end{proposition}
\begin{proof}
Let $E$ denote the full edge set of $G$. Recall that for every $S\in \mathcal L (G)$  the contraction $G/S$ is dual to $G^*[S^c]$, and the lattice of flats satisfies $\mathcal L(G/S)\simeq \mathcal L (G)_{\geq S}$. We can rewrite Equation \eqref{eq:BZ} as follows:
$$
\varphi_{G^*}(2)=
\sum_{H\in \mathcal L(G)^{op}} \mu_{\mathcal L(G)^{op}}(H,\hat{0}) E_{{G^*[H^c]}}(1) 
$$
where $\hat{0}$ denotes the minimal element of $\mathcal L(G)$, i.e., the set of all loops of $G$. Passing to the contraction,
$$
\varphi_{(G/S)^*}(2)=
\sum_{
\substack{H\in \mathcal L(G)^{op}\\
H\leq S}} \mu_{\mathcal L(G)^{op}}(H,\hat{0}) E_{{G^*[H^c]}}(1) 
$$
(where we used that $S^c\cap H^c=H^c$ as $S\subseteq H$).
Now M\"obius inversion on $\mathcal L(G)^{op}$ gives
$$
E_{{G^*[H^c]}}(1) =\sum_{
\substack{S\in \mathcal L(G)^{op}\\
S\leq H}}
\varphi_{(G/S)^*}(2)\, .
$$
Now, if $G/S$ is not bipartite it contains an odd cycle, and thus $(G/S)^*$ contains an odd cut $D$. Since our flows are ``circulations'' in the sense of \cite[\S 6.1]{Diestel}, the net flow across any cut is $0$. Since the net flow of any nowhere-zero $2$-flow on $(G/S)^*$ across the odd cut $D$ is a sum of an odd number of terms $\pm 1$, such a flow cannot exist, so $\varphi_{(G/S)^*}(2)=0$ if $G/S$ is not bipartite. Thus, setting $H=\hat{0}$, with Lemma \ref{lem:EE} and Proposition \ref{prop:facetsbipartite} we obtain the claim. 
\end{proof}

\begin{example}
Let $G=C_n$ be an $n$-cycle and label its edges with the set $[n]$. Then the elements of $\mathcal L(G)$ are $[n]$ itself and all subsets $S\subset [n]$, $\vert S \vert < n-1$. Accordingly, $G/S$ is either the graph with one vertex and no edges (whose symmetric edge polytope has $1$ facet) or an $(n-\vert S \vert)$-cycle. We obtain
$$
\# (\mathcal P_{C_n}^\triangle \cap \mathbb Z^d)
= 1 + \sum_{\substack{0\leq i<n-1\\ n-i \textrm{ even}}}
{n \choose i }{n-i \choose (n-i)/2}\, .
$$
\end{example}

\settocdepth{part}

\bibliographystyle{siam}
\bibliography{Xbib_arXiv}

\end{document}